\documentclass[reqno]{amsart}
\usepackage[left=1in,right=1in,top=1in,bottom=1in]{geometry}
\usepackage{tikz}

\usetikzlibrary{shapes,decorations,calc,arrows}
\usepackage[foot]{amsaddr}
\usepackage{amsthm}
\usepackage{amscd}
\usepackage{amsfonts}
\usepackage{amsmath}
\usepackage{amssymb}
\usepackage{mathrsfs}
\usepackage{multirow}
\usepackage{verbatim}
\usepackage{url}
\usepackage[hidelinks]{hyperref}
\usepackage{graphicx}
\usepackage{cite}
\usepackage{fancyhdr}
\usepackage{yfonts}
\usepackage{setspace}
\usepackage{titlesec}
\usepackage{enumitem}
\usepackage{dsfont}

\titleformat{\section}[hang]
{\normalfont\Large\bfseries}
{\thesection.}{0.5em}{}

\titlespacing*{\section}{0pc}{2pc}{0.25pc}

\titleformat{\subsection}[runin]
{\normalfont\large\bfseries}
{\thesubsection}{0.5em}{}

\titlespacing{\subsection}{0pc}{1.5pc}{0.5pc}

\newcommand{\Aut}{\text{Aut}}

\newcommand{\Z}{\mathbb{Z}}
\newcommand{\Q}{\mathbb{Q}}
\newcommand{\R}{\mathbb{R}}

\renewcommand{\H}{\mathcal{H}}
\newcommand{\B}{\mathcal{B}}

\newcommand{\K}{\mathcal{K}}

\newcommand{\vphi}{\varphi}

\newcommand{\Tr}{\text{Tr}}

\renewcommand{\>}{\right\rangle}

\newcommand{\dom}{\text{dom}}

\definecolor{ggreen}{HTML}{7FDD99}

\definecolor{obnoxious}{HTML}{662288}

\newcommand{\sltimes}[1]{\mathbin{_{#1}\ltimes}}

\newcommand{\E}{\mathcal{E}}
\newcommand{\Ind}[1]{\operatorname{Ind}#1}

\newcommand{\Sd}{\operatorname{Sd}}

\newtheorem{thm}{Theorem}[section]

\newtheorem{prop}[thm]{Proposition}
\newtheorem{lem}[thm]{Lemma}
\newtheorem{cor}[thm]{Corollary}

\theoremstyle{definition}
\newtheorem{defi}[thm]{Definition}
\newtheorem{ex}[thm]{Example}

\newtheorem{rem}[thm]{Remark}

\title{\textbf{Projective representations of almost unimodular groups}}
\author{Aldo Garcia Guinto}
\address{Department of Mathematics, Michigan State University, East Lansing, MI 48824, USA\hfill \url{garci575@msu.edu}}
\date{}

\begin{document}

\begin{abstract}
Given an almost unimodular $G$, so that the Plancherel weight $\varphi_G$ on the group von Neumann algebra $L(G)$ is almost periodic, we show that the basic construction for the inclusion $L(G)^{\varphi_G} \leq L(G)$ is isomorphic to a twisted group von Neumann algebra of $G \times \Delta_G(G)\hat{\ }$ with a continuous 2-cocycle, where $\Delta_G$ is the modular function. We show that when $G$ is second countable and admits a Borel 2-cocycle, $G$ is almost unimodular if and only if the central extension $\mathbb{T} \rtimes_{(1,\omega)} G$ is almost unimodular. Using this result and the connection between $\omega$-projective representations of $G$ and the representations of $\mathbb{T} \rtimes_{(1,\omega)} G$, we show that the formal degrees of irreducible and factorial square integrable projective representations behaved similarly to their representations counterparts and obtain the Atiyah--Schmid formula in the setting of second countable almost unimodular groups with a 2-cocycle twist and a finite covolume subgroup, which uses the Murray--von Neumann dimension for certain Hilbert space modules over the twisted group von Neumann algebra with its twisted Plancherel weight. \end{abstract}

\maketitle

\section*{Introduction}    
Given a locally compact group $G$ with a Borel 2-cocycle $\omega: G \times G \to \mathbb{T}$, an \textit{$\omega$-projective representation} ($\pi,\H)$ of $G$ is a map $\pi$ from $G$ into the set $U(\H)$ of all unitaries on $\H$ such that $\pi(s)\pi(t) = \omega(s,t) \pi(st)$ and the map $s \mapsto \|\pi(s)\xi\|$ is measurable for all $\xi \in \H$. In \cite{Mac58}, Mackey introduced the theory of projective representations of locally compact second countable groups to understand the theory of representations of certain group extensions. That is, for a certain normal subgroup $N$ of a locally compact second countable group $H$, the irreducible representation theory of $H$ restricts to the irreducible representations of $N$ and the irreducible projective representations of $H/N$. In particular, Mackey showed that representations of the central extension $\mathbb{T} \rtimes_{(1,\omega)} G$ are closely related to $\omega$-projective representations of $G$. Here \textit{second countable} is used to endow $\mathbb{T} \rtimes_{(1,\omega)} G$ with a locally compact topology and if $\omega$ is continuous, second countability is not necessary (see \cite[Theorem 7.1]{Mac57}). A natural example is the left regular $\omega$-projective representation $\lambda^\omega_G(s)$ on $L^2(G)$ defined by 
    \[
         (\lambda_G^{\omega}(s)f)(t) = \omega(t^{-1},s)f(s^{-1}t), \qquad s,t \in G, f \in L^2(G)
    \]
and $L_\omega(G):= \lambda_G^\omega(G)''$ is called the \textit{twisted group von Neumann algebra}. In \cite{Sut80}, Sutherland shows that $L_\omega(G)$ is independent of $\omega$ up to cohomology and admits a canonical faithful normal semifinite weight $\varphi^\omega_G$ associated to a left Haar measure $\mu_G$. In particular, if $\omega$ is cohomologous to the trivial 2-cocycle, the above is simply the Plancherel weight $\varphi_G$ on the group von Neumann algebra $L(G)$. We call $\varphi_G^\omega$ the \textit{twisted Plancherel weight}. There has been recent activity on projective representations on locally compact groups (see \cite{Ens22} and \cite{Yang25}) and (twisted) group von Neumann algebras of locally compact groups (see \cite{Vaes25}, \cite{GGN25}, \cite{Miy25}, and \cite{Mor25}). In this paper, we study the projective representations of almost unimodular groups.

In \cite{GGN25}, we studied a class of locally compact groups whose kernel of the modular function $\Delta_G$ is open, which we called \textit{almost unimodular}, and studied their permanence properties and representation theory. Furthermore, we showed that $G$ is almost unimodular if and only if the Plancherel weight $\varphi_G$ on $L(G)$ is almost periodic in the sense of Connes \cite{Con72} (see \cite[Theorem 2.1]{GGN25}). Actually, this is also equivalent to the twisted Plancherel weight being almost periodic. Indeed, Sutherland obtain an isomorphism of Hilbert spaces $L^2(L_\omega(G),\varphi^\omega_G)\cong L^2(G)$ which carries the modular operator $\Delta_{\varphi^\omega_G}$ of $\varphi^\omega_G$ to the modular function $\Delta_G$ of $G$ acting by pointwise multiplication. Thus the modular automorphism group $\sigma^{\varphi^\omega_G}\colon \R\curvearrowright L(G)$ is determined by $\Delta_G$:
    \[
        \sigma_t^{\varphi^\omega_G}(\lambda^\omega_G(f)) = \int_G \Delta_G(s)^{it} \lambda^\omega_G(s) f(s)\ d\mu_G(s) \qquad t\in \R, f \in L^1_\omega(G),
    \]
where $L_\omega^1(G)$ is the twisted convolution group algebra (see Subsection~\ref{subsec:Twisted_gp_vNa}). Here the operator $\lambda^\omega_G(f)$ is defined as 
    \[
        \lambda^\omega_G(f):= \int_G f(s) \lambda^\omega_G(s) d\mu_G(s) \in L_\omega(G). 
    \]
Hence $G$ is almost unimodular if and only if $\varphi^\omega_G$ is almost periodic for any 2-cocycle $\omega$ (see Proposition~\ref{prop:twisted_Plancherel_weight_ap}). Although almost unimodularity of $G$ can be witnessed by any of its twisted group von Neumann algebras, factoriality cannot. In \cite{Vaes25}, Vaes constructed a locally compact group that admits a 2-cocycle such that the group von Neumann algebra is a factor while the twisted group von Neumann algebra is not. More importantly, these examples contrast the case when the group has the discrete topology, since it is known that the group von Neumann algebra being a factor forces all the twisted group von Neumann algebras to also be factors. In this case, Kleppner gives a general characterization of factoriality of the twisted group von Neumann algebra (see \cite[Corollary 1]{Kle62}). In this paper, we give examples where the twisted group von Neumann algebra is a semifinite factor while the group von Neumann algebra is purely infinite and not a factor (see Example~\ref{ex:twisted_gp_vNa_factor_but_not_gp_vNa}). The examples we give arise by constructing a (continuous) 2-cocycle on $G \times \Delta_G(G)\hat{\ }$ such that the twisted group von Neumann algebra is isomorphic to the basic construction associated to the inclusion $L(G)^{\varphi_G}\leq L(G)$, whenever $G$ is almost unimodular (see Theorem~\ref{thm:basic_construction_twisted}). Note that $G$ does not need to be \textit{second countable} in this case. 

More precisely, Mackey showed that an $\omega$-projective representation $(\pi,\H)$ of $G$ can be lifted to a representation on the central extension $\mathbb{T}\rtimes_{(1,\omega)} G$ via $\pi_\omega(a,s) := a \pi(s)$, where $(a,s) \in \mathbb{T}\rtimes_{(1,\omega)}G$ and $G$ is second countable (see \cite[Section 2]{Mac58} and \cite[Corollary of Theorem 1]{KL72}). In this article, we show that $G$ is almost unimodular if and only if $\mathbb{T} \rtimes_{(1,\omega)}G$ is almost unimodular and we can identify $L_\omega(G)$ as a corner of $L(\mathbb{T} \rtimes_{(1,\omega)} G)$ such that the restriction of the Plancherel weight on $L(\mathbb{T}\rtimes_{(1,\omega)} G)$ to this corner is the twisted Plancherel weight on $L_{\omega}(G)$(see Proposition~\ref{prop:extension_of_almost_unimodularity} and \ref{prop:decomp_of_central_extension}). In light of these results, we adapt the results of \cite{GGN25} to the projective setting. 

In Section~\ref{sec:proj_reps}, we study the behavior of the formal degrees associated with irreducible and factorial square integrable projective representations of almost unimodular groups. In the case for non-unimodular groups, the formal degree associated to an irreducible square integrable representation, introduced by Duflo and Moore, is an unbounded operator and to a factorial square integrable representation, defined by Moore, is a normal semifinite weight (see \cite{DM76} and \cite{Moo77}). One could define formal degrees for the projective setting, but it can be shown that an $\omega$-projective representation $(\pi,\H)$ of $G$ is irreducible (resp. factorial) square integrable if and only if the representation $(\pi_\omega,\H)$ of $\mathbb{T} \rtimes_{(1,\omega)} G$ is irreducible (resp, factorial) square integrable. Thus the formal degrees associated to $(\pi_\omega,\H)$ is also associated to $(\pi,\H)$ (see \cite[Theorem 2]{Ani06} for irreducible case and Proposition~\ref{prop:notion_of_proj_rep_and_rep}.\ref{part:square_integrable_semi-invariant_weight} for factorial case). Furthermore, if $G$ is almost unimodular, we showed that the \textit{formal degree operator} is diagonalizable and the \textit{formal degree weight} is almost periodic (see Theorem~\ref{thm:sq_int_irred_proj_reps} and Theorem~\ref{thm:formal_degree_is_almost_periodic}). That is, we give a projective representation analogue of \cite[Theorem 4.2]{GGN25} and \cite[Theorem 4.5 and 4.6]{GGN25}. 

In Section~\ref{sec:Murray-von_Neumann_dim_twist}, we show how the Murray--von Neumann dimension of $(L_\omega(G),\varphi^\omega_G)$-modules scales for subgroups $H\leq G$ with \emph{finite covolume} and generalize the Atiyah--Schmid formula to the setting of second countable almost unimodular groups with 2-cocycles (see Theorem~\ref{thm:general_covolume_dimension_twisted} and \ref{thm:Atiya-schmidt_formula_twisted}). When $H=\Gamma$ is a lattice in a (necessarily unimodular) locally compact second countable group $G$ and $\omega$ is trivial, the formula recovers the Atiyah--Schmid formula (see \cite[Equation (3.3)]{AS77} and also \cite[Theorem 3.3.2]{GdlHJ89}). The Atiyah--Schmid formula, when $H$ is a lattice, was generalized to projective representations in \cite[Theorem 4.3]{Ens22} and projective tempered representations in \cite[Theorem 3.5]{Yang25}. Note that the twisted Plancherel weight of $L_\omega(H)$, when $H$ is a lattice, is the usual tracial state and our results are adaptations of the theorems in \cite[Subsection 5.3]{GGN25}.

\section*{Acknowledgment} I would like to thank my advisor, Brent Nelson, for his constant support, many helpful suggestions and discussions. I would also like to thank Jun Yang for discussions on projective representations.

\section{Preliminaries}
Throughout we let $G$ denote a locally compact group, which is always assumed to be Hausdorff. We will use lattice notation for the collection of closed subgroups of $G$. That is, we write $H\leq G$ to denote that $H$ is a closed subgroup of $G$, we write $H_1\vee H_2$ for the closed subgroup generated by $H_1,H_2\leq G$, etc. Similarly for a von Neumann algebra $M$ and its collection of (unital) von Neumann subalgebras and for a Hilbert space $\H$ and its collection of closed subspaces. All homomorphisms between groups are assumed to also be continuous and all isomorphisms are assumed to also be homeomorphisms. In particular, all representations on Hilbert spaces are assumed to be strongly continuous. A representation of a von Neumann algebra $M$ will always mean a normal unital $*$-homomorphism $\pi\colon M\to B(\H)$.

Throughout, $\mu_G$ will denote a \emph{left Haar measure} on $G$, which is a non-trivial, left translation invariant Radon measure. We follow the convention in \cite{Fol16} and take a \emph{Radon measure} to be a Borel measure that is finite on compact sets, outer regular on Borel sets, and inner regular on open sets. Such measures always exist and are unique up to scaling. Additionally, $\mu_G(U)>0$ for all non-empty open sets $U$, $\mu_G(K)<\infty$ for all compact sets $K$, and the inner regularity of $\mu_G$ extends to all $\sigma$-finite subsets (see \cite[Section 2.2]{Fol16} or \cite[Section 11]{HR79} for further details). For $1\leq p\leq \infty$ we denote by $L^p(G)$ the $L^p$-space of $G$ with respect to left Haar measures, and for a given left Haar measure $\mu_G$ we will write $\|f\|_{L^p(\mu_G)}$ for the associated $p$-norm, whereas $\|f\|_\infty$ is the unambiguous $\infty$-norm. We also denote by $\mathcal{B}(G)$ the Borel $\sigma$-algebra on $G$. Recall that the \emph{modular function} $\Delta_G\colon G\to \R_+$ is the continuous homomorphism (where $\R_+:=(0,\infty)$ has its multiplicative group structure) determined by $\mu_G(E\cdot s)= \Delta_G(s) \mu_G(E)$ for $s\in G$ and $E\in \B(G)$, where $\mu_G$ is any left Haar measure on $G$. This yields the following change of variables formulas that will be used implicitly in the sequel: 
    \[ 
        \int_Gf(st) d\mu_G(s) = \Delta_G(t)^{-1} \int_G f(s) d\mu_G(s) \qquad \text{ and } \qquad\int_G f(s^{-1}) d\mu_G(s) =  \Delta_G(s)^{-1}\int_G f(s) d\mu_G(s)
    \]
for $t\in G$ and $f\in L^1(G)$. We say $G$ is \emph{unimodular} if $\Delta_G\equiv 1$.

We assume that the reader has some familiarity with modular theory for weights on von Neumann algebras, complete details can be found in \cite[Chapter VIII]{Tak03} (see also \cite[Section 1]{GGLN25} for a quick introduction to these concepts). Given a faithful normal semifinite weight $\varphi$ on a von Neumann algebra $M$, we denote
    \begin{align*}
        \sqrt{\dom}(\varphi)&:=\{x\in M\colon \varphi(x^*x)<+\infty\}\\
        \dom(\varphi)&:=\text{span}\{x^*y\colon x,y\in \sqrt{\dom}(\varphi)\}.
    \end{align*}
We write $L^2(M,\varphi)$ for the completion of $\sqrt{\dom}(\varphi)$ with respect to the norm induced by the inner product
    \[
        \<x,y\>_\varphi:= \varphi(y^*x) \qquad x,y\in \sqrt{\dom}(\varphi).
    \]
The \emph{modular conjugation} and \emph{modular operator} for $\varphi$ will be denoted by $J_\varphi$ and $\Delta_\varphi$, respectively. The \emph{modular automorphism group} of $\varphi$, which we view as an action $\sigma^\varphi\colon \R\curvearrowright M$, is then defined by
    \[
        \sigma_t^\varphi(x):= \Delta_\varphi^{it} x \Delta_\varphi^{-it}.
    \]
Then the \emph{centralizer} of $\varphi$ is the fixed point subalgebra under this action and is denoted
    \[
        M^\varphi :=\{x\in M\colon \sigma_t^\varphi(x) = x \ \forall t\in \R\}.
    \]
We also recall that $x\in M^\varphi$ if and only if $xy,yx\in \dom(\varphi)$ with $\varphi(xy) = \varphi(yx)$ for all $y\in \dom(\varphi)$ (see \cite[Theorem VIII.2.6]{Tak03}). That is, $M^\varphi$ is the largest von Neumann algebra of $M$ on which $\varphi$ is tracial.

Recall that a weight $\varphi$ is \emph{strictly semifinite} if its restriction to $M^\varphi$ is semifinite; that is, if $\dom(\varphi)\cap M^\varphi$ is dense in $M^\varphi$ in the strong (or weak) operator topology. This property has several equivalent characterizations (see \cite[Lemma 1.2]{GGLN25}), but the most relevant to this article is the existence of a faithful normal conditional expectation $\E_\varphi\colon M\to M^\varphi$. After \cite{Con72}, we say a faithful normal semifinite weight $\varphi$ is \emph{almost periodic} if its modular operator $\Delta_\varphi$ is diagonalizable. We will write $\Sd(\varphi)$ for the point spectrum of $\Delta_\varphi$, so that
    \[
        \Delta_\varphi = \sum_{\delta\in \Sd(\varphi)} \delta 1_{\{\delta\}}(\Delta_\varphi)
    \]
whenever $\varphi$ is almost periodic. Note that an almost periodic weight is automatically strictly semifinite by \cite[Proposition 1.1]{Con74}.

\subsection{Twisted group von Neumann algebra}\label{subsec:Twisted_gp_vNa}
In this subsection, we remind the reader of the twisted group von Neumann algebra associated to a locally compact group with a 2-cocycle and its canonical weight (see \cite{Sut80}). We note that Sutherland assumed that $G$ is second countable, but this assumption is not necessary for the results we need.

Recall that a \emph{(Borel) 2-cocycle} $\omega: G \times G \to \mathbb{T}$ on $G$ is a Borel function satisfying
    \[
        \omega(s,t)\omega(st,r) = \omega(t,r)\omega(s,tr)  \qquad s,t,r \in G.
    \]
We say that a 2-cocycle is \emph{normalized} if $\omega(s,e)=\omega(e,s) = 1$ for all $s \in G$. Further, a 2-cocycle is \emph{fully normalized} if it is normalized and satisfies $\omega(s,s^{-1})=1$ for all $s \in G$. Note that the pointwise conjugate $\overline{\omega}$ of a fully normalized 2-cocycle $\omega$ is also fully normalized 2-cocycle and satisfies 
    \[ 
        \overline{\omega(s,t)}=\omega(t^{-1},s^{-1}).
    \]
Two 2-cocycles $\omega_1$ and $\omega_2$ are cohomologous, denoted by $\omega_1 \sim \omega_2$, if there exists a Borel map $\rho: G \to \mathbb{T}$ such that $\rho(e) =1$ and 
    \[
        \omega_1(s,t) = \overline{\rho(s)\rho(t)}\rho(st)\omega_2(s,t) \qquad s,t \in G. 
    \]
Any 2-cocycle $\omega$ is cohomologous to a fully normalized 2-cocycle (see \cite[Section 1]{Kle74}).

Let $\lambda^\omega_G,\rho^\omega_G\colon G\to B(L^2(G))$ be the \emph{left and right regular $\omega$-projective representations} of $G$:
    \[
        [\lambda^\omega_G(s)f](t)=\omega(t^{-1},s)f(s^{-1}t) \qquad \text{ and } \qquad [\rho^\omega_G(s)f](t)=\Delta_G(s)^{1/2}\omega(t,s)f(ts)
    \]
where $s,t\in G$ and $f\in L^2(G)$. The \emph{twisted group von Neumann algebra} of $G$ is the von Neumann algebra generated by its left regular $\omega$-projective representation, $L_\omega(G):=\lambda^\omega_G(G)''$. We denote the von Neumann algebra generated by its right regular $\omega$-projective representation by $R_\omega(G):=\rho^\omega_G(G)''$. By \cite[Theorem 3.8]{Sut80}, we have that $R_{\overline{\omega}}(G)=L_\omega(G)'\cap B(L^2(G))$. By \cite[Proposition 2.4]{Sut80}, any two cohomologous 2-cocycles generate the same twisted von Neumann algebra, and so we can (and will) assume that $\omega$ is a fully normalized 2-cocycle for the rest of this subsection. For any $f\in L^1_\omega(G)$,
    \begin{align}\label{eqn:L1_function_operator}
        \lambda^\omega_G(f):=\int_G  f(t)\lambda^\omega_G(t)\ d\mu_G(t)
    \end{align}
defines an element on $L_\omega(G)$ such that $\| \lambda_G^\omega(f)\| \leq \|f\|_{L^1(\mu_G)}$, where $L^1_\omega(G)$ is the same norm space as $L^1(G)$ but equipped with the twisted convolution
    \[
        [f*_\omega g](s) = \int_G \omega(t,t^{-1}s)f(t)g(t^{-1}s)\ d\mu_G(t)=\int_G \overline{\omega(t^{-1},s)}f(t)g(t^{-1}s)\ d\mu_G(t)
    \]
and the involution
    \[
        [f^\sharp](s) = \Delta_G(s)^{-1}\overline{f(s^{-1})}.
    \]
These operations can be found in \cite[Definition 2.7]{Sut80} or \cite[Section 2]{EL69}, the convolutions will agree with the former after a change of variables. The mapping $f\mapsto \lambda^\omega_G(f)$ gives a $*$-homomorphism from $L_\omega^1(G)$ to $L_\omega(G)$ and the $*$-subalgebra $\lambda^\omega_G(L^1_\omega(G))$ is dense in $L_\omega(G)$ in the strong (and weak) operator topology (see \cite[Theorem 2.14]{Sut80}). 

Let $B_{b,c}(G)$ be the space of equivalence classes of bounded compactly supported Borel functions, where the equivalence is given by equality $\mu_G$-almost everywhere. We endow $B_{b,c}(G)$ with the twisted convolution, involution and inner product coming from being a subset of $L^1_\omega(G) \cap L^2(G)$, which is well-defined by \cite[Lemma 2.8]{Sut80}. Then $B_{b,c}(G)$ is a left Hilbert algebra, which is dense in $L^2(G)$ and $\lambda_G^\omega(B_{b,c}(G))''= L_\omega(G)$ (see \cite[Theorem 3.3]{Sut80}). We call the elements in the full Hilbert algebra associated to $B_{b,c}(G)$ by \textit{twisted left convolvers}; that is $f \in L^2(G)$ is a left convolver if $f *_\omega g \in L^2(G)$ and there exists a constant $k >0$ such that $\|f *_\omega g \|_{L^2(\mu_G)} \leq k \|g\|_{L^2(\mu_G)}$ for each $g \in L^2(G)$. The \emph{twisted Plancherel weight} associated to $\mu_G$ is a faithful normal semifinite weight defined on $L_\omega(G)_+$ by
    \[
        \varphi^\omega_G(x^*x) := \begin{cases}
            \|f\|_{L^2(\mu_G)}^2 & \text{if }x=\lambda^\omega_G(f)\text{, with $f$ a twisted left convolver} \\
            +\infty & \text{otherwise}
        \end{cases}.
    \]
In particular, we have $L^2( L_\omega(G), \varphi_G^\omega) = L^2(G)$ and the modular objects associated to the twisted Plancherel weight are 
    \begin{equation*}
        (S_{\varphi_G^\omega}f)(s)= \Delta_G(s)^{-1} \overline{f(s^{-1})}, \qquad
        (\Delta_{\varphi_G^\omega}f)(s)=\Delta_G(s)f(s),\qquad
        (J_{\varphi_G^\omega}f)(s)=\Delta_G(s)^{-1/2}\overline{f(s^{-1})},
    \end{equation*}
where $f \in L^2(G)$ and $s \in G$ (see \cite[Lemma 3.6]{Sut80}). Thus the modular automophism group ${\sigma^{\varphi^\omega_G}_t}$ is completely determined by $\Delta_G$:
    \[  
        \sigma^{\varphi^\omega_G}_t(\lambda_G^\omega(s)) = \Delta_G(s)^{it}\lambda_G^\omega(s) \qquad s \in G.
    \]  

Notice that if $\omega$ is continuous, then one can consider instead the compactly supported continuous functions with the product coming from the twisted convolution. When the 2-cocycle is cohomologous to the trivial 2-cocycle, this collapses to the group von Neumann algebra with its Plancherel weight.

\section{Almost unimodular groups with a twist}
The following result is the twisted case of \cite[Theorem 2.1]{GGN25}. In the introduction, we mentioned how a twisted Plancherel weight $\varphi^\omega_G$ on $L_\omega(G)$ is equivalent to $G$ being almost unimodular (see \cite[Definition 2.2]{GGN25}). By \cite[Proposition 1.1]{Con74}, we know that almost periodic implies strictly semifinite and the converse, in the case for $L_\omega(G)$, follows by \cite[Lemma 1.4]{GGLN25} and using the generators $\lambda_G^\omega(s)$ for $s \in G$. The other claims follow similarly to \cite[Theorem 2.1]{GGN25}, so we omit the proof. Note that the following two results do not need the assumption of second countability.

\begin{prop}\label{prop:twisted_Plancherel_weight_ap}
Let $G$ be a locally compact group equipped with a left Haar measure $\mu_G$, let $\omega:G \times G \to \mathbb{T}$ be a 2-cocycle, and let $\vphi^\omega_G$ be the associated twisted Plancherel weight on $L_\omega(G)$. Denote the modular function by $\Delta_G\colon G\to \R_+$. The following are equivalent:
    \begin{enumerate}[label=(\roman*)]
        \item $G$ is almost unimodular; \label{part:open_unimodular_part}

        \item $\varphi^\omega_G$ is strictly semifinite; \label{part:strictly_semifinite}

        \item $\varphi^\omega_G$ is almost periodic;\label{part:almost_periodic}
        
        \item $\Delta_G$ viewed as an operator affiliated with $L^\infty(G) \subset B(L^2(G))$ is diagonalizable. \label{part:mod_functional_diagonlizable}
    \end{enumerate}
In this case one has
    \[
        \Sd(\varphi^\omega_G)=\Delta_G(G) \qquad \text{ and } \qquad L_\omega(G)^{\vphi^\omega_G} = \{\lambda^\omega_G(s)\colon s\in \ker{\Delta_G}\}''\cong L_\omega(\ker{\Delta_G}).
    \]
Under the identification $L_\omega(G)^{\varphi^\omega_G} \cong L_\omega(\ker{\Delta_G})$, $\varphi^\omega_G|_{L_\omega(G)^{\varphi^\omega_G}}$ is the Plancherel weight on $L_\omega(\ker{\Delta_G})$ corresponding to the left Haar measure $\mu_G|_{\mathcal{B}(\ker{\Delta_G})}$.
\end{prop}

Recall that for a locally compact group $G$, let $\hat{\iota}\colon \R\to \Delta_G(G)\hat{\ }$ be the map dual to the inclusion map $\iota\colon \Delta_G(G) \hookrightarrow \R_+$; that is,
    \begin{equation*}
        (\hat{\iota}(t)\mid \delta) = ( t \mid  \iota(\delta)) = \delta^{it} \qquad t\in \R,\ \delta\in \Delta_G(G).
    \end{equation*}
When $G$ is almost unimodular, so that $\varphi_G$ is almost periodic, the modular automorphism group $\sigma^{\varphi_G}\colon \R\curvearrowright L(G)$ admits an extension $\alpha\colon \Delta_G(G)\hat{\ } \curvearrowright L(G)$ satisfying $\alpha_{\hat{\iota}(t)} = \sigma_t^{\varphi_G}$, called the \emph{point modular extension} of the modular automorphism group,  for all $t\in \R$ (see \cite[Section 1.4]{GGLN25} for more details). Furthermore, the basic construction for the inclusion $L(G)^{\varphi_G} \leq L(G)$ can be expressed as a crossed product associated to the action $\alpha : \Delta_G(G)\hat{\ }  \curvearrowright L(G)$, where $\varphi_G$ is the Plancherel weight on $L(G)$ (see \cite[Theorem 5.1]{GGN25}). We show that it can also be expressed as a twisted group von Neumann algebra, giving us a way to construct non-trivial 2-cocycles.

\begin{thm}\label{thm:basic_construction_twisted}
Let $G$ be an almost unimodular group, let $\varphi_G$ be a Plancherel weight on $L(G)$ and let $\alpha\colon \Delta_G(G)\hat{\ }\curvearrowright L(G)$ be the point modular extension of $\sigma^{\varphi_G}\colon \R\curvearrowright L(G)$.  Then 
    \[
        \langle L(G), e_{\varphi_G}\rangle \cong L(G) \rtimes_\alpha \Delta_G(G)\hat{\ } \cong L_\omega(\Delta_G(G)\hat{\ }\times G),
    \]
where $\omega$ is a continuous 2-cocycle on $\Delta_G(G)\hat{\ } \times G$ defined by 
    \[
        \omega( (\gamma_1,s_1),(\gamma_2,s_2)) =\overline{( \gamma_2 | \Delta_G(s_1))},
    \]
where $(\cdot\mid \cdot)\colon \Delta_G(G)\hat{\ }\times \Delta_G(G)\to \mathbb{T}$ is the dual pairing and $(\gamma_1,s_1), (\gamma_2,s_2) \in \Delta_G(G)\hat{\ }\times G.$ Furthermore, under the second identification, the dual weight of the Plancherel weight on $L(G)$ is the twisted Plancherel weight on $L_\omega(\Delta_G(G)\hat{\ }\times G)$. 
\end{thm}
\begin{proof}
Denote $K := \Delta_G(G)\hat{\ }$. The first equality comes from \cite[Theorem  5.1]{GGN25}, with $e_{\varphi_G}$ being mapped to $p_K:=\int_K \lambda_K(\gamma) d\mu_K(\gamma)$, where $\mu_K$ is the normalized Haar measure on $K$. We show that the latter equality holds. Recall that $L(G)\rtimes_\alpha K$ is generated by the operators $(\lambda_K(\gamma)f)(\gamma_1,t) = f(\gamma^{-1}\gamma_1,t)$ and $[\pi_\alpha(\lambda_G(s))f](\gamma_1,t) = \overline{(\gamma_1\mid \Delta_G(s))}f(\gamma_1,s^{-1}t)$, where $(\gamma_1, t) \in K \times G$, $s \in G$, $\gamma \in K$, and $ f \in L^2(K \times G)$. Thus the map that sends $\lambda_{K\times G}^\omega(\gamma,s)$ to $\lambda_{K}(\gamma)\pi_\alpha(\lambda_G(s))$ defines the desired normal unital $*$-isomorphism.

Let $\varphi_{K \times G}^\omega$ be a twisted Plancherel weight on $L_\omega(K \times G)$ and let $\widetilde{\varphi_G}$ be a dual weight on $L(G) \rtimes_\alpha K$ associated to a fixed Plancherel weight $\varphi_G$ on $L(G)$. Then $\psi:=\widetilde{\varphi_G} \circ \Phi$ is a faithful normal semifinite weight on $L_\omega(K \times G)$, where $\Phi: L_\omega(K \times G) \to L(G) \rtimes_\alpha K$ is the isomorphism defined above. The claim that $\psi$ agrees with $\varphi_{K \times G}^\omega$ follows by \cite[Proposition VIII.3.15]{Tak03}. That is, $\psi \circ \sigma^{\widetilde{\varphi_G}}_t = \psi$ for all $t \in \R$ and they agree in a $\sigma$-weakly dense $*$-subalgebra of $L_\omega(K \times G)$. For each $\gamma \in K$, the Connes' cocycle derivative $D(\varphi_G \circ \alpha_\gamma : \varphi_G)$ is trivial, since $\varphi_G \circ \alpha_\gamma = \varphi_G$. Then for each $t \in \R$, we have $\sigma^{\widetilde{\varphi_G}}_t(\pi_\alpha(x)) = \pi_\alpha(\sigma^{\varphi_G}(x))$ for all $x \in L(G)$ and $\sigma^{\widetilde{\varphi_G}}_t(\lambda_K(\gamma)) = \lambda_K(\gamma)$ for all $\gamma \in K$ (see \cite[Theorem X.1.17]{Tak03}). Thus, $\psi \circ \sigma^{\widetilde{\varphi_G}}_t = \psi$ for all $t \in \R$ by applying $\Phi$ and $\Delta_{K\times G}(\gamma, s) = \Delta_G(s)$ for all $(\gamma, s) \in K \times G$. Consider the following $\sigma$-weakly dense $*$-subalgebra $\text{span}\{\lambda_{K \times G}^\omega(f_1\otimes g_1)^*\lambda_{K \times G}^\omega(f_2\otimes g_2): f_1,f_2 \in B_{b,c}(G), g_1,g_2 \in K_{b,c}(K) \}$, where $(f\otimes g )(\gamma, s)= f(\gamma)g(s)$ for all $(\gamma,s) \in K \times G$. By \cite[Theorem X.1.17]{Haa79}, the unique operator valued weight $T: L(G) \rtimes_\alpha K \to \pi_\alpha(L(G))$ such that $\widetilde{\varphi_G} (x) = \varphi_G \circ \pi_\alpha^{-1}(T(x))$ for all $x \in (L(G) \rtimes_\alpha K)_+$ is implemented by $Tx = \sum_{\delta \in \Delta_G(G)} \hat{\alpha}_\delta(x)$, where $\hat{\alpha}$ is the dual action of $\Delta_G(G)$ on $L(G) \rtimes_\alpha K$. For each $\delta \in \Delta_G(G)$, the dual action $\hat{\alpha}_\delta$ satisfies $\hat{\alpha}_\delta(\pi_\alpha(x)) = \pi_\alpha(x)$ and $\hat{\alpha}_\delta(\lambda_K(\gamma)) = \overline{(\gamma | \delta)} \lambda_K(\gamma)$ for all $x \in L(G)$ and $\gamma \in K$. Thus for $f_1,f_2\in B_{b,c}(G), g_1,g_2 \in K_{b,c}(K)$, we obtain
    \begin{align*}
        \psi\left(\lambda_{K \times G}^\omega(f_1\otimes g_1)^*\lambda_{K \times G}^\omega(f_2\otimes g_2)\right) &= \widetilde{\varphi_G}\left(\pi_\alpha(\lambda_G(g_1)^*)\lambda_K(f_1)^*\lambda_K(f_2)\pi_\alpha(\lambda_G(g_2)) \right)\\
        &=\varphi_G \circ \pi_\alpha^{-1}\left(\pi_\alpha(\lambda_G(g_1)^*)T\left(\lambda_K(f_1^\sharp *f_2)\right)\pi_\alpha(\lambda_G(g_2))\right)\\
        &= \langle f_2,f_1  \rangle_{L^2(\mu_K)} \varphi_G\left(\lambda_G(g_1)^*\lambda_G(g_2)\right)\\
        &= \langle f_2,f_1  \rangle_{L^2(\mu_K)}\langle g_2,g_1  \rangle_{L^2(\mu_G)} = \varphi_{K\times G}^\omega\left(\lambda_{K \times G}^\omega(f_1\otimes g_1)^*\lambda_{K \times G}^\omega(f_2\otimes g_2)\right),    
    \end{align*}
as claimed.
\end{proof}

We now give examples of groups admitting a 2-cocycle such that the twisted group von Neumann algebra is a factor, but the group von Neumann algebra is not. This is different from the examples that Vaes constructed since those examples are of locally compact groups admitting a 2-cocycle such that the group von Neumann algebra is a factor but not the twisted group von Neumann algebra (see \cite{Vaes25}).

\begin{ex}\label{ex:twisted_gp_vNa_factor_but_not_gp_vNa}
Let $G$ be an almost unimodular group, let $\Delta_G(G)\hat{\ }$ be the dual group of $\Delta_G(G)$ and $\omega$ be the 2-cocycle defined in Theorem~\ref{thm:basic_construction_twisted}. If $L(\ker\Delta_G)$ is a factor, then $L_\omega(\Delta_G(G)\hat{\ }\times G)$ is also a factor by \cite[Theorem 5.1]{GGN25} and Theorem~\ref{thm:basic_construction_twisted}. But $L(\Delta_G(G)\hat{\ } \times G) \cong L(\Delta_G(G)\hat{\ }) \bar\otimes L(G)$ is not a factor since $\Delta_G(G)\hat{\ }$ is an abelian group. We now give explicit examples where $L(\ker\Delta_G)$ is a factor (see \cite[Example 5.3.2 and 5.3.3]{GGN25}). 

Let $\alpha \colon \text{GL}_2(\R) \curvearrowright \R^2$ be the action by matrix multiplication. We consider a countable intermediate subgroup $\text{SL}_2(\Z)\leq H_1\leq \text{GL}_2(\R)$ and restricting $\alpha$ to $H_1$, we set $G_1:= H_1 \sltimes{\alpha} \R^2$. By \cite[Example 5.3.2]{GGN25}, we have that $G_1$ is almost unimodular, $L(\ker{\Delta_{G_1}})$ is a separable non-hyperfinite factor of type $\mathrm{II}_\infty$ and $L(G_1)$ is a separable non-injective factor of type: $\mathrm{II}_\infty$ if $\det(H_1)=\{1\}$;  $\mathrm{III}_\lambda$ if $\det(H_1)=\lambda^\Z$ for some $0< \lambda <1$; and $\mathrm{III}_1$ if $\det(H_1)$ is dense in $\R_+$. In particular, $L_\omega(\Delta_{G_1}(G_1)\hat{\ }\times G_1)$ is a semifinite factor and $L(\Delta_{G_1}(G_1)\hat{\ }\times G_1)$ is a purely infinite non-factor if $\det(H_1)$ is non-trivial.

Let $\text{UT}_2(\R)$ denote the upper triangular matrices with real entries and let $\alpha \colon \text{UT}_2(\R) \curvearrowright \R^2$ be the action by matrix multiplication. Consider a countable intermediate subgroup $\text{N}_2(\Q)\leq H_2\leq \text{UT}_2(\R)$, where $\text{N}_2(\Q)=\text{UT}_2(\R) \cap \text{SL}_2(\Q)$. Restricting $\alpha$ to $H_2$, we set $G_2:= H_2 \sltimes{\alpha} \R^2$. By \cite[Example 5.3.3]{GGN25}, we have that $G_2$ is almost unimodular, $L(\ker{\Delta_{G_2}})$ is a separable hyperfinite factor of type $\mathrm{II}_\infty$ and $L(G_2)$ is the unique separable injective factor of type: $\mathrm{II}_\infty$ if $\det(H_2)=\{1\}$; $\mathrm{III}_\lambda$ if $\det(H_2)=\lambda^\Z$ for some $0< \lambda <1$; and $\mathrm{III}_1$ if $\det(H_2)$ is dense in $\R_+$. In particular, $L_\omega(\Delta_{G_2}(G_2)\hat{\ }\times G_2)$ is a semifinite factor and $L(\Delta_{G_2}(G_2)\hat{\ }\times L(G_2)$ is a purely infinite non-factor when $\det(H_2)$ is non-trivial. $\hfill\blacksquare$
\end{ex}

To move the results in \cite[Section 4 and Subsection 5.3]{GGN25} to the twisted case, we remind the reader of the central extension associated to a Borel 2-cocycle $\omega$ of a locally compact second countable group $G$ to a compact abelian second countable group $A$ written additively. Note that from here on out, we assume that $G$ is second countable.

Recall that a normalized Borel 2-cocycle of $G$ over $A$ is a map $\omega: G \times G \to  A$ satisfying $\omega(e,s) =\omega(s,e) =0$ and the cocycle identity
    \[
        \omega(s,t)+\omega(st,r) = \omega(t,r)+\omega(s,tr),
    \]
where $s,t,r \in G$. In the case that $\omega$ is a 2-cocycle over $\mathbb{T}$, we will just say it is a 2-cocycle. We consider the Borel cocycle semidirect product $A\rtimes_{(1,\omega)} G$ associated to the Borel cocycle action $(1,\omega)\colon G \curvearrowright A$, where $1: G \to \Aut(A)$ is the trivial map (see \cite[Section 2]{Mac58} and \cite[Section 3]{EL69}). This Borel group is the set $A \times G$ provided with the following operations 
    \[
        (a,s) (b,t) = (a+b+\omega(s,t),st) \qquad \text{ and } \qquad (a,s)^{-1} = (-a-\omega(s,s^{-1}),s^{-1}),
    \]
and admits a left invariant measure $\mu_A \times \mu_G$, where $(a,s),(b,t) \in A \rtimes_{(1,\omega)}G$, $\mu_A$ is a normalized measure on $A$ and $\mu_G$ is a left Haar measure on $G$. By \cite[Theorem 7.1]{Mac57}, there exists a unique locally compact topology on $A\rtimes_{(1,\omega)} G$ such that it becomes a locally compact group and its left Haar measure is $\mu_A \times \mu_G$. This topology is called the Weil's topology (see \cite[Section 1]{Kle74} and \cite[Section 62]{Hal50}). In particular, the modular function in this case is 
    \begin{equation}\label{eqp:modular_function_of_central_extension}
        \Delta_{A\rtimes_{(1,\omega)} G}(a,s)= \Delta_G(s) \quad (a,s)\in A\rtimes_{(1,\omega)} G.
    \end{equation}
It follows that $G$ is non-unimodular if and only if  $A\rtimes_{(1,\omega)} G$ is non-unimodular. Identifying $A \cong A \times e$ (as locally compact groups) gives us that $A$ is a compact normal subgroup of $A\rtimes_{(1,\omega)} G$. Then $A\rtimes_{(1,\omega)} G$ is a topological central extension of $A$ by $G$, since $A$ is in the center of the extension. We also have that $(A\rtimes_{(1,\omega)} G)/A \cong G$ (as locally compact groups). If the 2-cocycle over $A$ is continuous, then the topology on the cocycle semidirect product is the product topology (see \cite[Lemma 2]{Kle74}). 

The following results further establish the connection between the structure of $A \rtimes_{(1,\omega)} G$ and $G$ equipped with a Borel 2-cocycle $\omega$ over $A$.

\begin{prop}\label{prop:extension_of_almost_unimodularity}
Let $\omega$ be a normalized 2-cocycle of a locally compact group $G$ over a compact abelian group $A$. Then $G$ is almost unimodular if and only if $A\rtimes_{(1,\omega)} G$ is almost unimodular, whenever $\omega$ is continuous or $G$ and $A$ are second countable. 
\end{prop}
\begin{proof}
when the group is second countable, almost unimodularity is equivalent to the image of the modular function being countable (see  \cite[Proposition 2.3]{GGN25}). Thus the claim follows, when $A$ and $G$ are second countable, by equation (\ref{eqp:modular_function_of_central_extension}), since $\Delta_{A\rtimes_{(1,\omega)} G}(A\rtimes_{(1,\omega)} G)=\Delta_G(G)$. In the case that $\omega$ is continuous, the claim follows from the topology on $A\rtimes_{(1,\omega)} G$ being the product topology.
\end{proof}

For each $\gamma \in A\hat{\ } $, we define a projection $p_\gamma$ on $L^2(A\rtimes_{(1,\omega)} G)$ by
    \begin{equation}\label{eqn:projections_onto_the_twist}
        (p_{G,\gamma} f)(a,s) := \int_A (\gamma\mid b)[\lambda_{A\rtimes_{(1,\omega)} G}(b,e)^* f](a,s) d\mu_A(b) \qquad f \in L^2(G), (a,s) \in A\rtimes_{(1,\omega)} G,
    \end{equation}
where $(\cdot\mid\cdot)$ denote dual pairings between locally compact abelian groups. Since $A$ is in the center of $A\rtimes_{(1,\omega)} G$, we have that $p_{G,\gamma}$ is a central projection in $L(A\rtimes_{(1,\omega)} G)$. By a change of variables, $f \in p_{G,\gamma} L^2(A\rtimes_{(1,\omega)} G)$ if $f(a,s) = \overline{(\gamma\mid a)}f(0,s)$. We define a unitary $U_{G,\gamma}:p_{G,\gamma} L^2(A\rtimes_{(1,\omega)} G) \to  L^2(G)$ via 
    \begin{equation}\label{eqn:unitary_of_the_extension}
        (U_{G,\gamma} f)(s):=f(0,s)
    \end{equation} 
and it satisfies
    \begin{equation}\label{eqn:equivalence_of_twisted_and_extension}
        U_{G,\gamma}  \lambda_{A\rtimes_{(1,\omega)} G}(a,s) = (\gamma \mid a)\lambda_G^{\gamma \omega}(s) U_{G,\gamma} \qquad (a,s) \in A\rtimes_{(1,\omega)} G, \gamma \in A\hat{\ },
    \end{equation}
where $\gamma\omega(s,t)= (\gamma | \omega(s,t))$ is a 2-cocycle and we extent it to a partial isometry on $L^2(A\rtimes_{(1,\omega)} G)$. Note that $\{p_{G,\gamma}\}_{\gamma \in A\hat{\ }}$ is a pairwise orthogonal family of projection that sum up to 1. The following result is a decomposition of the group von Neumann algebra of $A\rtimes_{(1,\omega)} G$ analogous to \cite[Theorem 2, Theorem 3]{EL69} and \cite[Theorem 5.3]{Edw69}. Although this is well-known to experts, we provide the proof for completeness.

\begin{prop}\label{prop:decomp_of_central_extension}
Let $\omega$ be a normalized 2-cocycle of a locally compact second countable group $G$ over a compact second countable abelian group $A$. Then there exists a family of pairwise orthogonal central projections $\{p_{G,\gamma}\}_{ \gamma \in A\hat{\ }}$ summing up to 1, defined in equation (\ref{eqn:projections_onto_the_twist}), such that 
    \[
        L(A\rtimes_{(1,\omega)} G) = \bigoplus_{\gamma \in A\hat{\ }} L(A\rtimes_{(1,\omega)} G)p_{G,\gamma},
    \]
and for each $\gamma \in A\hat{\ }$, we have a normal, unital $*$-homomorphism $\Omega_{G,\gamma} :L(A\rtimes_{(1,\omega)} G) \to  L_{\gamma\omega}(G)$ satisfying
    \[
        \Omega_{G,\gamma}(\lambda_{A\rtimes_{(1,\omega)} G}(a,s)) = (\gamma|a)\lambda^{\gamma\omega}_G(s) \qquad  (a,s)\in A\rtimes_{(1,\omega)} G,
    \]
where $\gamma\omega(s,t)= (\gamma | \omega(s,t))$ is a 2-cocycle of $G$. Furthermore, for each $\gamma \in A\hat{\ }$, $\Omega_{G,\gamma}$ restricted to $ L(A\rtimes_{(1,\omega)} G)p_{G,\gamma}$ is an isomorphism and $\varphi_{A\rtimes_{(1,\omega)} G}\circ \Omega_{\gamma}^{-1}$ is the twisted Plancherel weight on $L_{\gamma\omega}(G)$, where $\varphi_{A\rtimes_{(1,\omega)} G}$ is the Plancherel weight on $L(A\rtimes_{(1,\omega)} G)$.
\end{prop}
\begin{proof}
Denote $G(\omega):=A\rtimes_{(1,\omega)} G$. The first equality follows from our discussion above. 

Fixing $\gamma \in A\hat{\ }$, we let $p_{G,\gamma}$ and $U_{G,\gamma}$ be as in equation (\ref{eqn:projections_onto_the_twist}) and (\ref{eqn:unitary_of_the_extension}), respectively. Thus, setting     \[
        \Omega_{G,\gamma}(x) := U_{G,\gamma} xU^*_{G,\gamma},
    \]
we have a normal unital $*$-homomorphism $\Omega_{G,\gamma}: L(G(\omega)) \to L_{\gamma\omega}(G)$, which satisfies $\Omega_{G,\gamma}(\lambda_{G(\omega)}(a,s)) = (\gamma|a)\lambda^{\gamma\omega}_G(s)$ for all $(a,s) \in G(\omega)$. It follows that $\Omega_{G,\gamma}$ restricted to $L(G(\omega))p_{G,\gamma}$ is an isomorphism.

Finally, we prove that for a fixed $\gamma \in A\hat{\ }$, the normal semifinite weight $\psi:= \varphi_{G(\omega)}\circ \Omega_{G,\gamma}^{-1}$ on $L_{\gamma\omega}(G)$ agrees with the twisted Plancherel weight $\varphi^{\gamma\omega}_G$, where $\varphi_{G(\omega)}$ is the Plancherel weight on $L(G(\omega))$. Since $\Omega_{G,\gamma}^{-1} \circ \sigma_t^{\varphi^{\gamma\omega}_G} =\sigma_t^{\varphi_{G(\omega)}} \circ \Omega_{G,\gamma}^{-1}$ for all $t \in \R$, it follows that $\psi$ commutes with $\varphi_G^{\gamma\omega}$, that is $\psi \circ \sigma_t^{\varphi^{\gamma\omega}_G} = \psi$ for all $t \in \R$. By \cite[Proposition VIII.3.15]{Tak03}, we obtain our claim if we show that they agree on the $\sigma$-weakly dense $*$-subalgebra $B = \text{span}\{ \lambda^{\gamma\omega}_G(f)^*\lambda_G^{\gamma\omega}(g) : f,g \in B_{b,c}(G)\}$ in $\sqrt{\dom}(\varphi^{\gamma\omega}_G)$, which is invariant over the automorphism group $\sigma^{\varphi^{\gamma\omega}_G}$. Now for any $f \in B_{b,c}(G)$, the normality of $\Omega_{G,\gamma}^{-1}$ gives us
    \begin{align*}
        \Omega_{G,\gamma}^{-1}(\lambda_G^{\gamma\omega}(f)) &= \int_G f(s) \Omega_{G,\gamma}^{-1}(\lambda^{\gamma\omega}_G(s))d\mu_G(s)= \int_G\int_A \overline{(\gamma \mid a)}f(s) \lambda_{G(\omega)}(a,s)d\mu_A(a)d\mu_G(s) \\
        &= \int_G\int_A (U^*_{G,\gamma} f)(a,s) \lambda_{G(\omega)}(a,s)d\mu_A(a)d\mu_G(s)= \lambda_{G(\omega)}(U^*_{G,\gamma} f),
    \end{align*}
where the third equality follows from $(U^*_{G,\gamma} f)(a,s)= \overline{(\gamma | a)}f(s)$ for $f \in L^2(G)$ and $(a,s) \in G(\omega)$. Additionally, for any $\gamma' \in A\hat{\ }$, and $g \in p_{G,\gamma'}L^2(G(\omega))$, we have
    \begin{align*}
        [\lambda_{G(\omega)}(U^*_{G,\gamma} f)g](a,s) &= \int_G\int_A (U^*_{G,\gamma} f)(b,t) [\lambda_{G(\omega)}(b,t)g](a,s) d\mu_A(b)d\mu_G(t)\\
        &=\int_G\int_A \overline{(\gamma\mid b)}f(t)g(a-b-\omega(t,t^{-1})+\omega(t^{-1},s),t^{-1}s) d\mu_A(b)d\mu_G(t)\\
        &=\int_G\int_A \overline{(\gamma\mid b)}f(t) \overline{(\gamma'\mid a-b-\omega(t,t^{-1})+\omega(t^{-1},s))}g(0,t^{-1}s) d\mu_A(b)d\mu_G(t)\\
        &=\overline{(\gamma' \mid a)} \int_G\int_Af(t) (\gamma'\mid \omega(t,t^{-1}s))g(0,t^{-1}s)d\mu_G(t)\int_A \overline{(\gamma\mid b)} (\gamma'\mid b)d\mu_A(b)\\
        &= \overline{(\gamma' \mid a)}\int_G f(t)[\lambda^{\gamma' \omega}_G(t) (U_{G,\gamma'}g)](s) d\mu_G(t) \int_A \overline{(\gamma\mid b)} (\gamma'\mid b)d\mu_A(b),
    \end{align*}
which will be zero if $\gamma \neq \gamma'$ or finite if $\gamma = \gamma'$ and so $U^*_{G,\gamma} f \in L^2(G(\omega))$ is a left convolver for any $f \in B_{b,c}(G)$. Thus, for any $f,g \in B_{b,c}(G)$, we obtain using that fact that $\lambda_{G(\omega)}(U_{G,\gamma}^* g) p_{G,\gamma'} = 0$ in the second equality,
    \begin{align*}
        \psi(\lambda_G^{\gamma\omega}(f)^*\lambda_G^{\gamma\omega}(g)) &=\varphi_{G(\omega)}\left(\lambda_{G(\omega)}(U_{G,\gamma}^*f)^* \lambda_{G(\omega)}(U_{G,\gamma}^* g)p_{\gamma}\right)\\
        &= \sum_{\gamma' \in A\hat{\ }} \varphi_{G(\omega)}\left(\lambda_{G(\omega)}(U_{G,\gamma}^*f)^* \lambda_{G(\omega)}(U_{G,\gamma}^* g)p_{G,\gamma'}\right) \\
        &= \langle U^*_{G,\gamma} g, U^*_{G,\gamma} f\rangle_{L^2(\mu_A \times \mu_G)}\\
        &= \langle g, f\rangle_{L^2(\mu_G)}= \varphi_G^{\gamma\omega}\left(\lambda_G^{\gamma\omega}(f)^*\lambda_G^{\gamma\omega}(g)\right),
    \end{align*}         
and the equality on $B$ follows from linearity of $\psi$.
\end{proof}

Lastly, we show that we can identify $L_\omega(H)$ as a subalgebra of $L_\omega(G)$ such that $\lambda_H^\omega(h) \mapsto \lambda_G^\omega(h)$ for all $h \in H$ using Proposition~\ref{prop:decomp_of_central_extension}, which is most likely known to experts.

\begin{prop} \label{prop:identification_of_twisted_subgroup_vNa}
Let $H$ be a closed subgroup of a locally compact second countable group $G$ and $\omega: G \times G \to \mathbb{T}$ be a 2-cocycle. Then 
    \[
        L_\omega(H) \cong \{\lambda_G^\omega(s) : s \in H\}''\subset L_\omega(G).
    \]  
\end{prop}
\begin{proof}
Assume that $\omega$ is fully normalized (see discussion before \cite[Lemma 2]{Kle74} and \cite[Proposition 2.4]{Sut80}). We denote $G(\omega) := \mathbb{T} \rtimes_{(1,\omega)}G$ and $H(\omega) := \mathbb{T} \rtimes_{(1,\omega)}H$. By \cite[Proposition 2.8]{HR19}, we have a $*$-isomorphism $\Psi:L(H(\omega)) \to L(G(\omega))$ by $\Psi(\lambda_{H(\omega)}(a,s)) = \lambda_{G(\omega)}(a,s)$. By the normality of $\Psi$, it follows that $\Psi(\lambda_{H(\omega)}(1,s)p_{H,1}) = \lambda_{G(\omega)}(1,s)p_{G,1}$, where $p_{H,1}$ and $p_{G,1}$ are defined by equation (\ref{eqn:projections_onto_the_twist}) associated to $H(\omega)$ and $G(\omega)$, respectively. Thus, using the $*$-isomorphisms $\Omega_{H,1}:L(H(\omega))p_{H,1} \to  L_\omega(H)$ and $\Omega_{G,1}: L(G(\omega))p_{G,1}\to L_\omega(G) $ in Proposition \ref{prop:decomp_of_central_extension}, we have $\Omega_{G,1}(\Psi( \Omega_{H,1}^{-1}(\lambda^\omega_H(s)))) = \lambda^\omega_G(s)$.
\end{proof}

\section{Square integrable projective representations}\label{sec:proj_reps}
In this section, we show the relation between the square integrable $\omega$-projective representations of a second countable almost unimodular groups $G$ and the square integrable $\omega$-projective representations of $\ker\Delta_G$, where $\omega: G\times G \to \mathbb{T}$ is a 2-cocycle. When $\omega$ is cohomologous to the trivial 2-cocycle, we obtain the results in \cite[Section 4]{GGN25}.

Given an $\omega$-projective representation $(\pi,\H)$ of $G$, the map $\pi_\omega:\mathbb{T} \rtimes_{(1,\omega)} G \to U(\H)$ defined by
    \begin{equation}\label{eqn:proj_rep_extension}
        \pi_\omega(a,s)\xi := a \pi(s)\xi, \quad \xi \in \H, (a,s) \in \mathbb{T}\rtimes_{(1,\omega)}G,
    \end{equation}
is a representation of $\mathbb{T} \rtimes_{(1,\omega)} G$ (see \cite[22.20]{HR79} and \cite[Theorem 1]{Kle74}). If a representation $(\rho, \K)$ of $\mathbb{T} \rtimes_{(1,\omega)} G$ satisfies that the restriction of $\rho$ to the closed normal subgroup $\mathbb{T} \,(= \mathbb{T} \times e)$ is one dimensional, then we have an $\omega$-projective representation $(\pi, \K)$ of $G$ defined by 
    \[
        \pi(s)\xi := \rho(1,s)\xi \qquad \xi \in \H_\omega,s \in G 
    \] 
and the representation $(\pi_\omega,\K)$ of $\mathbb{T}\rtimes_{(1,\omega)} G$ defined by equation (\ref{eqn:proj_rep_extension}) and $(\rho,\K)$ agree (see \cite[Theorem 2.1]{Mac58}). 

The theory of induced projective representation of closed subgroups is deeply connected to the induced representation of closed subgroups of the central extension. Given a closed subgroup $H$ of $G$, we have that $\omega$ is a 2-cocycle on $H$ and we identify $\mathbb{T}\rtimes_{(1,\omega)} H$ with the closed subgroup generated by $1 \times H$ and $\mathbb{T}\times e$ in $\mathbb{T}\rtimes_{(1,\omega)}G$. The latter is closed since $\mathbb{T}\rtimes_{(1,\omega)} G/\mathbb{T} \cong G$. For an $\omega$-projective representation $(\pi_1,\H_1)$ of $H$, the induced representation of $\mathbb{T} \rtimes_{(1,\omega)} G$ associated to the representation of $\mathbb{T} \rtimes_{(1,\omega)}H$ given by equation (\ref{eqn:proj_rep_extension}) of $(\pi_1,\H_1)$ satisfies that its restriction to the subgroup $\mathbb{T}$ is one dimensional (see \cite[Theorem 12.1]{Mac52} with $G_1 = \mathbb{T} \rtimes_{(1,\omega)} H$ and $G_2 = \mathbb{T}$). Thus by \cite[Theorem 2.1]{Mac58}, there exists an $\omega$-projective representation $(\pi, \H)$ of $G$, which can be shown to be unitary equivalence to the induced $\omega$-projective representation $(\pi_1,\H_1)$ of $H$ (see \cite[Theorem 4.1]{Mac58} and its preceding discussion). The following is the projective representation analogue of \cite[Theorem 4.1]{GGN25}.

\begin{thm}\label{thm:proj_reps_induced_from_kernel}
Let $(\pi, \H)$ be an $\omega$-projective representation of a second countable almost unimodular group $G$ with a 2-cocycle $\omega: G \times G \to \mathbb{T}$. The following are equivalent:
\begin{enumerate}[label=(\roman*)]
    \item $(\pi,\H)\cong \Ind^G_{\ker{\Delta_G}}(\pi_1,\H_1)$ for some $\omega$-projective representation $(\pi_1,\H_1)$ of $\ker{\Delta_G}$.

    \item There exists a representation $U: \Delta_G(G)\hat{\ } \to U(\H)$ satisfying $U_\gamma \pi(s) U_\gamma^* = (\gamma \mid \Delta_G(s)) \pi(s)$, where $(\cdot\mid \cdot)\colon \Delta_G(G)\hat{\ }\times \Delta_G(G)\to \mathbb{T}$ is the dual pairing.
\end{enumerate}
\end{thm}
\begin{proof}
Assume that $\omega$ is fully normalized (see discussion before \cite[Lemma 2]{Kle74} and \cite[Proposition 2.4]{Sut80}). Denote $G_1 := \ker\Delta_G$, $G(\omega):= \mathbb{T} \rtimes_{(1,\omega)} G$ and $G_1(\omega):= \mathbb{T}\rtimes_{(1,\omega)} G_1$. Let $(\pi_\omega,\H)$ be the representation of the almost unimodular group $G(\omega)$ defined by equation (\ref{eqn:proj_rep_extension}) associated to $(\pi,\H)$. \\

\noindent \textbf{(i)$\Rightarrow$(ii):} Let $(\pi_{1,\omega}, \H_1)$ be the representation of $G_1(\omega)$ defined by equation (\ref{eqn:proj_rep_extension}) associated to $(\pi_1,\H_1)$. By the discussion preceding this theorem, it follows that $(\pi_\omega,\H)\cong \Ind^{G(\omega)}_{G_1(\omega)}(\pi_{1,\omega},\H_1)$. Since $G(\omega)$ is almost unimodular with $\Delta_{G(\omega)}(G(\omega))= \Delta_G(G)$, we apply \cite[Theorem 4.1]{GGN25} to obtain the existence of a representation $U: \Delta_G(G)\hat{\ } \to U(\H)$ satisfying $U_\gamma \pi_\omega(a,s) U_\gamma^* = (\gamma \mid \Delta_G(s)) \pi_\omega(a,s)$. The claim follows from $\pi(s)=\pi_\omega(1,s) $.\\

\noindent \textbf{(ii)$\Rightarrow$(i):} Let $U$ be a representation of $\Delta_G(G)\hat{\ }$. Then $U$ also satisfies $U_\gamma \pi_\omega(a,s) U_\gamma^* = (\gamma \mid \Delta_G(s)) \pi_\omega(a,s)$. Again, since $G(\omega)$ is almost unimodular, we apply \cite[Theorem 4.1]{GGN25} to obtain the existence of a representation $(\rho,\H_1)$ of $G_1(\omega)$ such that $\text{Ind}_{G_1(\omega)}^{G(\omega)}(\rho, \H_1)\cong (\pi_\omega, \H)$. Since $\Delta_{G(\omega)}(G(\omega)) = \Delta_G(G)$, we obtain $\H= \bigoplus_{\delta \in \Delta_G(G)} \pi_{\omega}(\sigma(\delta))\H_1$, where $\sigma: \Delta_G(G) \to G(\omega)$ is a normalized section of $\Delta_G(G)$; that is, $\sigma(1) = (1,e)$ and $\Delta_{G(\omega)}\circ \sigma (\delta)= \delta $ for all $\delta \in \Delta_{G(\omega)}(G(\omega))$ (see proof of \cite[Theorem 4.1]{GGN25}). In particular, $(\pi_{\omega}|_{G_1(\omega)},\H_1)$ is equal to $(\rho,\H_1)$ as representations of $G_1(\omega)$. This implies that $\rho$ restricted to $\mathbb{T}$ is one-dimensional. By \cite[Theorem 2.1]{Mac58}, there is an $\omega$-projective representation $(\pi_1,\H_1)$ of $G_1$ such that the representation $(\pi_{1,\omega},\H_1)$ of $G_1(\omega)$ defined by equation (\ref{eqn:proj_rep_extension}) is unitarily equivalent to $(\rho,\H_1)$. Thus $\Ind^G_{G_1}(\pi_1,\H_1) \cong (\pi,\H)$ as claimed.
\end{proof}

\subsection{Irreducible square integrable projective representations} Let $(\pi,\H)$ be an irreducible $\omega$-projective representation of a locally compact second countable group $G$ with a 2-cocycle $\omega$ and let $(\pi_\omega,\H)$ denote the representation of $\mathbb{T}\rtimes_{(1,\omega)}G$ defined in (\ref{eqn:proj_rep_extension}), which is irreducible. Recall that $(\pi,\H)$ is called \textit{square integrable} if there exists  $\xi,\eta \in \H$ such that the function $c_{\xi,\eta}(s) := \langle \pi(s) \xi,\eta\rangle$, called a coefficient of $\pi$, is a non-zero element in $L^2(G)$. By \cite[Theorem 2]{Ani06}, $(\pi,\H)$ is square integrable if and only if $(\pi_\omega,\H)$ is square integrable. Additionally, $(\pi,\H)$ is subequivalent to the left regular $\omega$-projective representation of $G$ and the formal degree operator associated to $(\pi, \H)$ is also associated to $(\pi_\omega,\H)$. The formal degree operator $D$ associated to a representation (or $\omega$-projective representation) $(\pi ,\H)$ of $G$ is a unique non-zero, self-adjoint, positive operator $D$ on $\H$ satisfying
    \[
        \pi(s)D\pi^*(s) =\Delta_G(s)^{-1}D \qquad s \in G.
    \]
The above property of $D$ is referred to as \emph{semi-invariance} with weight $\Delta_G^{-1}$ (see \cite[Section 1]{DM76}). Furthermore, $\xi \in \dom(D^{-1/2})$ if and only if $c_{\xi,\eta}\in L^2(G)$ for all $\eta\in \H$, and the following orthogonality relation holds
    \begin{align}\label{eqn:formal_degree_operator_equation}
        \langle c_{\xi_1,\eta_1}, c_{\xi_2,\eta_2}\rangle_{L^2(\mu_G)} = \int_{G} \langle\pi(s) \xi_1,\eta_1\rangle \overline{\langle\pi(s) \xi_2,\eta_2\rangle} d\mu(s) = \langle D^{-\frac12}\xi_1,D^{-\frac12} \xi_2\rangle \langle \eta_2, \eta_1\rangle,
    \end{align}
for $\xi_1,\xi_2 \in \dom(D^{-1/2})$ and $ \eta_1,\eta_2 \in \H$ (see \cite[Theorem 3]{DM76}).
If $G$ is unimodular, the formal degree operator is just the formal degree associated to $\omega$-projective representation $(\pi,\H)$ of $G$ or representation $(\pi_\omega,\H)$ of $\mathbb{T} \rtimes_{(1,\omega)} G$ (see \cite[Theorem 2]{Ani06}). The following is the projective analogue of \cite[Theorem 4.2]{GGN25}. 

\begin{thm}\label{thm:sq_int_irred_proj_reps}
Let $(\pi,\H)$ be an irreducible square integrable $\omega$-projective representation of a second countable almost unimodular group $G$ with a 2-cocycle $\omega: G \times G \to \mathbb{T}$. Then there exists an irreducible square integrable $\omega$-representation $(\pi_1,\H_1)$ of $\ker\Delta_G$ such that $(\pi,\H) \cong \emph{Ind}_{\ker\Delta_G}^G(\pi_1,\H_1)$, and the formal degree operator $D$ for $(\pi,\H)$ is diagonalizable with 
    \[
        D=\sum_{\delta \in \Delta_G(G)}(d_{\pi_1}\delta)1_{d_{\pi_1}\delta}(D),
    \]
where $d_{\pi_1}$ is the formal degree of $(\pi_1,\H_1)$.
\end{thm}
\begin{proof}
Assume that $\omega$ is fully normalized (see discussion before \cite[Lemma 2]{Kle74} and \cite[Proposition 2.4]{Sut80}). Denote $G_1 := \ker\Delta_G$, $G(\omega):= \mathbb{T} \rtimes_{(1,\omega)} G$ and $G_1(\omega):= \mathbb{T}\rtimes_{(1,\omega)} G_1$. Let $(\pi_\omega,\H)$ be the representation of $G(\omega)$ given by (\ref{eqn:proj_rep_extension}). By \cite[Theorem 4.2]{GGN25}, there exists an irreducible square integrable representation $(\rho, \H_1)$ of $G_1(\omega)$ such that $(\pi_\omega,\H) \cong \text{Ind}_{G_1(\omega)}^{G(\omega)}(\rho,\H_1)$. Since $\Delta_{G(\omega)}(G(\omega)) = \Delta_G(G)$, we fixed normalized section $\sigma: \Delta_G(G) \to G$ of $\Delta_G(G)$ and define a normalized section $\sigma^\omega: \Delta_G(G) \to G(\omega)$ of $\Delta_G(G)$ by $\sigma^\omega(\delta):=(1, \sigma(\delta))$ to obtain $\H= \bigoplus_{\delta \in \Delta_G(G)} \pi_{\omega}(\sigma^\omega(\delta))\H_1=\bigoplus_{\delta \in \Delta_G(G)} \pi(\sigma(\delta))\H_1$ (see \cite[Theorem 4.2]{GGN25}. In particular, $(\pi_{\omega}|_{G_1(\omega)},\H_1)$ is equal to $(\rho,\H_1)$ as representations of $G_1(\omega)$ and the former is one-dimensional when restricted to $\mathbb{T}$. By \cite[Theorem 2.1]{Mac58}, we have an $\omega$-projective representation $(\pi_1,\H_1)$ of $G_1$ such that its representation of $G_1(\omega)$, using equation (\ref{eqn:proj_rep_extension}), is unitary equivalent to $(\rho,\H)$. By \cite[Theorem 2]{Ani06}, the formal degree operator associated to $(\pi_\omega,\H)$ is the same formal degree operator associated to $(\pi,\H)$. Similarly for the formal degree of $(\rho, \H_1)$ and $(\pi_1,\H_1)$. Thus by \cite[Theorem 4.2]{GGN25}, we obtain that $D$ is diagonalizable with the above formula.  
\end{proof}

\begin{rem}
Observe that the decomposition of $\H$ in the proof of the previous theorem implies that for fixed $\xi\in \H_1\setminus\{0\}$, the map
    \begin{align*}
        v\colon \H &\to L^2(G)\\
            \eta &\mapsto \frac{1}{d_{\pi_1}^{1/2} \|\xi\|} c_{\xi,\eta}
    \end{align*}
defines an isometry satisfying $v\pi(s) = \lambda^\omega_G(s) v$ for all $s\in G$. Since $\pi(s)^*\eta \in \H_1$ if and only if $\Delta_G(s) = \delta$ for any $\eta\in \pi(\sigma(\delta)) \H_1$, it follows that $v [\pi(\sigma(\delta)) \H_1] = 1_{\{\delta\}}(\Delta_G)v$ and hence $vD = d_{\pi_1} \Delta_G$.$\hfill\blacksquare$
\end{rem}

\subsection{Factorial square integrable projective representations} In this subsection, we give the projective representation analogue of the results in \cite[Subsection 4.2]{GGN25}. Towards that end, we move all the characterizations of square integrability for factorial representations in \cite{Moo77} to the setting of projective representations (see also \cite[Proposition 2.3(a)]{Ros78}). We begin by reminding and moving some of the definitions to the projective setting.

Recall that an $\omega$-projective representation $(\pi,\H)$ of a locally compact second countable group $G$ with a 2-cocycle $\omega$ is said to be \textit{factorial} (or a \textit{factor} $\omega$-projective representation) if $\pi(G)''$ is a factor. In this case, $(\pi,\H)$ is called \textit{square integrable} if there exists vectors $\xi,\eta \in \H$ whose coefficient $c_{\xi,\eta}(s) = \langle \pi(s)\xi,\eta\rangle$ defines a non-zero element in $L^2(G)$ (see \cite[Section 2]{Ros78}). We say that two $\omega$-projective representations $(\pi_1,\H_1)$ and $(\pi_2,\H_2)$ are \textit{quasi-equivalent} if there exists an isomorphism $\Phi: \pi_1(G)'' \to \pi_2(G)''$ such that $\Phi(\pi_1(s)) = \pi_2(s)$. 

In \cite{Moo77}, Moore defined the notion of formal degree as a (normalized) semi-invariant weight $\psi$ of degree $\Delta_G$ on the von Neumann algebra $\pi(G)''$ such that there exists some non-zero $x \in \dom(\psi)$ satisfying $s \mapsto \psi(\pi(s)x)$ is a square integrable function on $G$. The definition of semi-invariant for weights on a von Neumann algebra associated to projective representations is defined the same as in \cite{Moo77}, but we mention it here for the convenience of the reader.

\begin{defi}
Let $(\pi,\H)$ be a factorial $\omega$-projective representation of a locally compact second countable group $G$ with a 2-cocycle $\omega: G \times G \to \mathbb{T}$ and let $\Delta: G \to \R_+$ be a continuous group homomorphism of $G$. A non-zero normal semifinite weight $\psi$ on $\pi(G)''$ is said to be semi-invariant of degree $\Delta$ if 
    \[
    \psi(\pi(s) x\pi(s)^*) =\Delta(s)\psi(x), 
    \]
where  $s \in G, x \in (\pi(G)'')_+.\hfill \blacksquare$
\end{defi}

The following result shows us that the notions of quasi-equivalence, square integrability and the existence of a semi-invariant weight associated to an $\omega$-projective representation of $G$ can be lifted to their original definitions but associated to the representation of $\mathbb{T} \rtimes_{(1,\omega)} G$ defined by (\ref{eqn:proj_rep_extension}).

\begin{prop}\label{prop:notion_of_proj_rep_and_rep}
Let $G$ be a locally compact second countable group with a 2-cocycle $\omega:G \times G \to \mathbb{T}$. Then
\begin{enumerate}[label=(\alph*)]
    \item $(\pi,\H)$ is a factorial square integrable $\omega$-projective representation of $G$ if and only if $(\pi_\omega,\H)$ is a factorial square integrable representation of $\mathbb{T} \rtimes_{(1,\omega)}G$, where $\pi_\omega(a,s) :=a\pi(s)$ for each $(a,s) \in \mathbb{T}\rtimes_{(1,\omega)}G$. Furthermore, the semi-invariant weight of degree $\Delta_G$ on $\pi(G)''$ is also a semi-invariant weight of degree $\Delta_{\mathbb{T}\rtimes_{(1,\omega)}G}$ on $\pi_\omega(\mathbb{T}\rtimes_{(1,\omega)}G)''$;\label{part:square_integrable_semi-invariant_weight}
    \item $(\pi,\H)$ is quasi-equivalent to $(\rho,\K)$ as $\omega$-projective representations of $G$ if and only if $(\pi_\omega,\H)$ is quasi-equivalent to $(\rho_\omega,\K)$ as representations of $\mathbb{T}\rtimes_{(1,\omega)} G$; \label{part:quasi-equivalence}
\end{enumerate}
\end{prop}
\begin{proof}
Assume that $\omega$ is fully normalized (see discussion before \cite[Lemma 2]{Kle74} and \cite[Proposition 2.4]{Sut80}). Denote $G(\omega):= \mathbb{T}\rtimes_{(1,\omega)}G$.

\noindent\textbf{(a)}: By definition of $\pi_\omega$, we have that $\pi_\omega(G(\omega))'' = \pi(G)''$. It follows that $\pi_\omega$ is factorial if and only if $\pi$ is factorial. Then for $\xi,\eta \in H$, we have 
    \[
        c_{\xi,\eta}(s) = \langle \pi(s)\xi, \eta\rangle =  \overline{a} \langle \pi_\omega(a,s)\xi,\eta\rangle=\overline{a}c_{\xi,\eta}(a,s),
    \]
where $(a,s) \in G(\omega)$. Since $\mathbb{T}$ is compact, it follows that $\pi$ is square integrable if and only if $\pi_\omega$ is square integrable. The last claim follows from $\Delta_{G(\omega)}(G(\omega)) = \Delta_G(G)$.\\

\noindent \textbf{(b)}: This follows from the fact that $\pi_\omega(G(\omega))\cong \pi(G)$ and $\rho_\omega(G(\omega)) \cong \rho(G)$.
\end{proof}

Thus we obtain the projective representation analogue of \cite[Theorem 3]{Moo77} (see also \cite[Proposition 2.(a)]{Ros78}). 

\begin{cor}
Let $(\pi,\H)$ be a factorial $\omega$-projective representation of a locally compact second countable group $G$ with a 2-cocycle $\omega:G \times G \to \mathbb{T}$. Then the following are equivalent:
\begin{enumerate}[label=(\roman*)]
    \item $(\pi,\H)$ is square integrable;
    \item $(\pi,\H)$ is quasi-equivalent to an $\omega$-projective subrepresentation of the left regular $\omega$-projective representation of $G$.
    \item the von Neumann algebra associated to $(\pi,\H)$ admits a semi-invariant weight $\psi$ of degree $\Delta_G$ such that there exists some non-zero $x \in \dom(\psi)$ satisfying $s \mapsto \psi(\pi(s)x)$ is a square integrable function on $G$.
\end{enumerate}
\end{cor}

By \cite[Proposition 2.1 and Theorem 1]{Moo77}, the semi-invariant weight $\psi$ of $\Delta_G$ in Proposition \ref{prop:notion_of_proj_rep_and_rep}.\ref{part:square_integrable_semi-invariant_weight} is faithful and unique up to scaling. By \cite[Theorem 4]{Moo77}, we can always normalize $\psi$ to satisfy 
    \[
        \int_G \psi(\pi(s) x) \overline{\psi(\pi(s)y)} d\mu_G(s) = \psi(y^*x) \qquad x,y \in \dom(\psi),
    \]
and in this case $\psi$ is called the \textit{formal degree} of $(\pi,\H)$ (and $(\pi_\omega,\H)$). Note that if $G$ is unimodular, the formal degree is a tracial weight. 

Finally we prove the projective analogue of \cite[Theorem 4.5]{GGN25} and \cite[Theorem 4.6]{GGN25}. That is, the formal degree associated to the projective representation of an almost unimodular group is almost periodic. 

\begin{thm}\label{thm:formal_degree_is_almost_periodic}
Let $(\pi, \H)$ be a factorial square integrable $\omega$-projective representation of a second countable almost unimodular group $G$ with a 2-cocycle $\omega:G \times G \to \mathbb{T}$. Then the formal degree $\psi$ of $(\pi,\H)$ is almost periodic with $(\pi(G)'')^\psi=\pi(\ker{\Delta_G})''$.
\end{thm}
\begin{proof}
Assume that $\omega$ is fully normalized (see discussion before \cite[Lemma 2]{Kle74} and \cite[Proposition 2.4]{Sut80}). Let $(\pi_\omega,\H)$ be the representation of $\mathbb{T} \rtimes_{(1,\omega)} G$ given by equation (\ref{eqn:proj_rep_extension}). Then by Proposition \ref{prop:notion_of_proj_rep_and_rep}.\ref{part:square_integrable_semi-invariant_weight}, the formal degree $\psi$ of $(\pi,\H)$ is also the formal degree of $(\pi_\omega,\H)$. Hence by \cite[Theorem 4.5]{GGN25}, we have 
    \[
        (\pi(G)'')^\psi\cong (\pi_\omega(\mathbb{T} \rtimes_{(1,\omega)}G)'')^\psi=\pi_\omega(\mathbb{T} \rtimes_{(1,\omega)}\ker{\Delta_G})''\cong\pi(\ker{\Delta_G})'',
    \]
as claimed.
\end{proof}

\begin{thm}\label{thm:sq_int_factor_proj_reps}
Let $G$ be a second countable almost unimodular group with a 2-cocycle $\omega:G \times G \to \mathbb{T}$ and let $(\pi_1, \H_1)$ be a factorial square integrable $\omega$-projective representation of $\ker\Delta_G$. Then $\emph{Ind}_{\ker \Delta_G}^G(\pi_1,\H_1)$ is also factorial and square integrable.
\end{thm}
\begin{proof}
Assume that $\omega$ is fully normalized (see discussion before \cite[Lemma 2]{Kle74} and \cite[Proposition 2.4]{Sut80}). Denote $G_1:=\ker\Delta_G$, $G(\omega):=\mathbb{T} \rtimes_{(1,\omega)}G$ and $G_1(\omega):=\mathbb{T} \rtimes_{(1,\omega)}G_1$. Let $(\pi_{1,\omega}, \H_1)$ be the representation of $G_1(\omega)$ given by equation (\ref{eqn:proj_rep_extension}). By Proposition \ref{prop:notion_of_proj_rep_and_rep}.\ref{part:square_integrable_semi-invariant_weight}, we have that $(\pi_{1,\omega}, \H_1)$ is a factorial square integrable representation of $G_1(\omega)$. Then by \cite[Theorem 4.6]{GGN25}, $\text{Ind}_{G_1(\omega)}^{G(\omega)}(\pi_{1,\omega},\H_1)$ is also factorial and square integrable. Thus by \cite[Theorem 4.1]{Mac58} and Proposition \ref{prop:notion_of_proj_rep_and_rep}.\ref{part:square_integrable_semi-invariant_weight}, we have that the induced $\omega$-projective representation $\text{Ind}_{G_1}^G(\pi_1,\H_1)$ of $G$ is also a factorial and square integrable. 
\end{proof}

\section{Murray--von Neumann dimension with a twist}\label{sec:Murray-von_Neumann_dim_twist}
In this section, we move all the results from \cite[Subsection 5.3]{GGN25} to the twisted group von Neumann algebras case. Recall that a closed subgroup $H \leq G$ of a locally compact group has \textit{finite covolume} if the quotient space $G/H$ admits a finite (non-trivial) $G$-invariant Radon measure $\mu_{G/H}$. We normalize this measure in such a way that 
    \begin{equation}\label{eqn:quotient_Haar_measure_formula}
        \int_G fd\mu_G = \int_{G/H}\int_H f(st) d\mu_H(t)d\mu_{G/H}(sH) \quad f \in L^1(G),
    \end{equation}
where $\mu_G$ and $\mu_H$ are fixed left Haar measures on $G$ and $H$, respectively (see \cite[Theorem 2.49]{Fol16}). In this case, the \emph{covolume} of $(H, \mu_H) \leq (G ,\mu_G)$ is the quantity
    \[
        [\mu_G: \mu_H] := \mu_{G/H}(G/H).
    \]
If $H \leq G$ is a finite index inclusion and $\mu_H = \mu_G|_{\mathcal{B}(H)}$, then $\mu_{G/H}$ is the counting measure, which gives us $[\mu_G: \mu_H]= [G:H]$.

Given a closed subgroup $H$ of $G$ and a 2-cocycle $\omega$ on G, we have that $\omega$ is a 2-cocycle on $H$ and we identify $\mathbb{T}\rtimes_{(1,\omega)} H$ with the subgroup generated by $1 \times H$ and $\mathbb{T}\times e$ in $\mathbb{T}\rtimes_{(1,\omega)}G$. The generated subgroup is closed since $(\mathbb{T}\rtimes_{(1,\omega)} G)/\mathbb{T} \cong G$. It follows that $(\mathbb{T} \rtimes_{(1,\omega)}G)/(\mathbb{T} \rtimes_{(1,\omega)}H)\cong G/H$ and so $\mathbb{T} \rtimes_{(1,\omega)}H\leq \mathbb{T}\rtimes_{(1,\omega)}G$ is a finite covolume subgroup if and only if $H \leq G$ is a finite covolume subgroup. Additionally,  $[\mu_\mathbb{T} \times \mu_G : \mu_\mathbb{T} \times \mu_G]=[\mu_G:\mu_H]$, where $\mu_\mathbb{T}$ is a normalized Haar measure on $\mathbb{T}$. The following result is \cite[Proposition 5.7]{GGN25} associated to $\mathbb{T} \rtimes_{(1,\omega)} G$ and $G$.

\begin{prop}\label{prop:finite_covolume_subgroups_twisted} Let $G$ be a second countable almost unimodular group with finite covolume subgroup $H \leq G$ and $\omega: G\times G \to \mathbb{T}$ be a 2-cocycle. Then:
\begin{enumerate}[label=(\alph*)]
    \item $H$ and $\mathbb{T}\rtimes_{(1,\omega)} H$ are almost unimodular;
    \item $\Delta_{\mathbb{T} \rtimes_{(1,\omega)} G}|_{\mathbb{T}\rtimes_{(1,\omega)} H}=\Delta_{\mathbb{T} \rtimes_{(1,\omega)} H}$ and $\Delta_G|_H=\Delta_H$;
    \item $\ker\Delta_{\mathbb{T} \rtimes_{(1,\omega)} H}$ has finite covolume in $\ker\Delta_{\mathbb{T} \rtimes_{(1,\omega)} G}$;
    \item $\Delta_{\mathbb{T} \rtimes_{(1,\omega)} H}(\mathbb{T}\rtimes_{(1,\omega)} H)$ is a finite index subgroup of $\Delta_{\mathbb{T}\rtimes_{(1,\omega)} G}(\mathbb{T} \rtimes_{(1,\omega)} G)$.
\end{enumerate}
\end{prop}

Our goal is to apply the Murray--von Neumann dimension theory for strictly semifinite weights of \cite{GGLN25} to the twisted group von Neumann algebra with its twisted Plancherel weight associated to a second countable almost unimodular group equipped with a 2-cocycle. First, we show that the decomposition of $L(\mathbb{T}\rtimes_{(1,\omega)} G)$, as seen in Proposition \ref{prop:decomp_of_central_extension}, also holds for the basic construction of $\langle L(\mathbb{T}\rtimes_{(1,\omega)} G), \varphi_{\mathbb{T}\rtimes_{(1,\omega)} G}\rangle$, when $G$ is almost unimodular. Note that under the identification $L^2(L_\omega(G),\varphi^\omega_G)\cong L^2(G)$, we obtain $e_{\varphi^\omega_G} = 1_{\ker{\Delta_G}}$, or more generally 
    \[
        \lambda^\omega_G(s) e_{\varphi^\omega_G} \lambda^\omega_G(s)^* = 1_{\Delta_G^{-1}(\{\delta\})},
    \]
for any $s\in G$ with $\Delta_G(s)=\delta$.

\begin{lem} \label{lem:decomposition_of_basic_construction}
Let $G$ be a second countable almost unimodular group with a 2-cocycle $\omega: G \times G \to \mathbb{T}$. Then 
    \[
        \langle L(\mathbb{T}\rtimes_{(1,\omega)} G), \varphi_{\mathbb{T}\rtimes_{(1,\omega)} G}\rangle = \bigoplus_{n \in \Z} \langle L(\mathbb{T}\rtimes_{(1,\omega)} G), \varphi_{\mathbb{T}\rtimes_{(1,\omega)} G}\rangle p_n,
    \]
where $\{p_n\}_{n\in \Z}$ are central projections summing up to 1 defined in (\ref{eqn:projections_onto_the_twist}) and the homomorphism $\Omega_{G,n}$ of Proposition~\ref{prop:decomp_of_central_extension} can be extended to a normal, unital $*$-homomorphism $\widetilde{\Omega}_{G,n} :\langle L(\mathbb{T}\rtimes_{(1,\omega)} G), \varphi_{\mathbb{T}\rtimes_{(1,\omega)} G}\rangle \to  \langle L_{\omega^n}(G), e_{\varphi^{\omega^n}_G}\rangle $, such that 
    \[
        \widetilde{\Omega}_{G,n}(e_{\varphi_{\mathbb{T} \rtimes_{(1,\omega)} G}}) = e_{\varphi^{\omega^n}_G},
    \]
where $\omega^n(s,t)= \omega(s,t)^n$ is a 2-cocycle of $G$. Furthermore, for each $n \in \Z$, the restriction of $\widetilde{\Omega}_{G,n}$ on $\langle L(\mathbb{T}\rtimes_{(1,\omega)} G), \varphi_{\mathbb{T}\rtimes_{(1,\omega)} G}\rangle p_n$ is an isomorphism. 
\end{lem}
\begin{proof}
Assume that $\omega$ is fully normalized (see discussion before \cite[Lemma 2]{Kle74} and \cite[Proposition 2.4]{Sut80}). We denote $G_1:= \ker\Delta_G$, $G(\omega) := \mathbb{T} \rtimes_{(1,\omega)}G$ and $G_1(\omega) := \mathbb{T} \rtimes_{(1,\omega)}G_1$.

Since $G(\omega)$ is almost unimodular, $\varphi_{G(\omega)}$ is strictly semifinite (see \cite[Theorem 2.1]{GGN25}). Note that $e_{\varphi_{G(\omega)}}= 1_{G_1(\omega)}$ commutes with $p_{n}$ for each $n \in \Z$, since $p_{n} \in L(G(\omega))^{\varphi_{G(\omega)}}$. Thus the first claim follows. 

Now for any $\delta \in \Delta_G(G)$, $U_{n}:p_{n}L^2(\Delta^{-1}_{G(\omega)}(\{\delta\})) \to  L^2(\Delta_G^{-1}(\{\delta\}))$ is a unitary, where $U_{n}$ is defined by equation (\ref{eqn:unitary_of_the_extension}). It follows that $U_{n}p_{n}1_{\Delta^{-1}_{G(\omega)}(\{\delta\})} = 1_{\Delta^{-1}_G(\{\delta\})}$. In particular $U_{n}p_{n}e_{\varphi_{G(\omega)}} =1_{G_1}= e_{\varphi^{\omega^n}_G}$. Hence, the map defined by $\widetilde{\Omega}_{G,n}(x) := U_nxU^*_n$ for $x \in \langle L(G(\omega)), \varphi_{G(\omega)}\rangle$ is an extension of $\Omega_{G,n}$ and satisfies the claim.
\end{proof}

Recall that a pair $(\pi,\H)$ is called a \emph{left $(L_\omega(G),\varphi_G^\omega)$-module} if $\pi\colon \<L_\omega(G),e_{\varphi^\omega_G}\> \to B(\H)$ is a normal unital $*$-homomorphism (see \cite[Definition 2.7]{GGLN25}). For such a pair there always exists an ancillary Hilbert space $\K$ and an isometry $v\colon \H \to L^2(G)\otimes \K$ called a \emph{standard intertwiner} satisfying $v\pi(x) = (x\otimes 1)v$ for all $x\in \<L_\omega(G),\varphi^\omega_G\>$ (see \cite[Proposition 2.4]{GGLN25}). The \emph{Murray--von Neumann dimension} of $(\pi,\H)$ is defined as
    \[
        \dim_{(L_\omega(G),\varphi^\omega_G)}(\pi,\H):= (\varphi^\omega_G \otimes \Tr_{\K})\left[ (J_{\varphi^\omega_G}\otimes 1) vv^* (J_{\varphi^\omega_G}\otimes 1)\right],
    \]
and it is independent of $\K$ and $v$ (see \cite[Proposition 2.8]{GGLN25}). Next we prove the twisted group von Neumann algebra analogues of \cite[Theorem 5.8 and Corollary 5.9]{GGN25}.

\begin{thm}\label{thm:general_covolume_dimension_twisted}
Let $G$ be a second countable almost unimodular group with finite covolume subgroup $H \leq G$, let $\omega:G \times G \to \mathbb{T}$ be a 2-cocycle and let $\varphi^\omega_G$ (resp. $\varphi^\omega_H$) be the twisted Plancherel weight on $L_\omega(G)$ (resp. $L_{\omega}(H)$) associated to a left Haar measure $\mu_G$ on $G$ (resp. $\mu_H$ on $H$). For each set $\Delta$ of coset representatives of $\Delta_H(H) \leq \Delta_G(G)$, there exists a unique injective, normal, unital $*$-homomorphism $\theta^\omega_{\Delta} : \langle L_\omega(H),\varphi_H^\omega\rangle \to \langle L_\omega(G),\varphi_G^\omega\rangle$ satisfying 
    \[
        \theta_\Delta^\omega(\lambda^\omega_H(t))= \lambda^\omega_G(t)\quad t\in H, \qquad \text{ and } \qquad \theta^\omega_\Delta(e_{\varphi^\omega_H})=\sum_{\delta\in \Delta} 1_{\Delta_G^{-1}(\{\delta\})}.
    \]
Moreover, if $(\pi,\H)$ is a left $(L_\omega(G),\varphi^\omega_G)$-module, then $(\pi\circ \theta^\omega_\Delta,\H)$ is a left $(L_\omega(H),\varphi^\omega_H)$-module with
    \begin{align}\label{eqn:covolume_dimension_formula}
        \dim_{(L_\omega(H),\varphi^\omega_H)}(\pi\circ\theta^\omega_{\Delta},\H) = \left(\frac{1}{|\Delta|}\sum_{\delta\in \Delta} \delta\right) [\mu_G: \mu_H] \dim_{(L_\omega(G),\varphi^\omega_G)}(\pi,\H).
    \end{align}
\end{thm}
\begin{proof}
Assume that $\omega$ is fully normalized (see discussion before \cite[Lemma 2]{Kle74} and \cite[Proposition 2.4]{Sut80}). We denote $G_1:= \ker\Delta_G$, $H_1:=\ker\Delta_H$ and for any subgroup $N$ of $G$, we denote $N(\omega) := \mathbb{T} \rtimes_{(1,\omega)}N$.

By Proposition~\ref{prop:finite_covolume_subgroups_twisted}, $H(\omega)$ is a finite covolume subgroup of the almost unimodular group $G(\omega)$ and we identifying $G(\omega)/ H(\omega) \cong G/H$ via $(a,s)H(\omega) \mapsto sH$, where $(a,s) \in G(\omega)$. Applying \cite[Theorem 5.8]{GGN25} to $H(\omega) \leq G(\omega)$, there exists a unique injective, normal, unital $*$-homomorphism $\theta_\Delta : \langle L(H(\omega)), e_{\varphi_{H(\omega)}} \rangle \to \langle L(G(\omega)), e_{\varphi_{G(\omega)}}\rangle$ satisfying 
    \[
        \theta_\Delta(\lambda_{H(\omega)}(a,s))= \lambda_{G(\omega)}(a,s)\quad (a,s)\in H(\omega), \qquad \text{ and } \qquad \theta_\Delta(e_{\varphi_{H(\omega)}})=\sum_{\delta\in \Delta} 1_{\Delta_{G(\omega)}^{-1}(\{\delta\})}.
    \] 
Notice that for each $n \in \Z$, we have $\theta_\Delta(p_{H,n}) = p_{G,n}$, where $p_{G,n}$ and $p_{H,n}$ are the central projections defined by equation (\ref{eqn:projections_onto_the_twist}) associated to $G(\omega)$ and $H(\omega)$, respectively. Thus, the injective, normal, unital $*$-homomorphism defined by
    \[
        \theta_\Delta^\omega := \widetilde{\Omega}_{G,1} \circ \theta_\Delta \circ \widetilde{\Omega}_{H,1}^{-1}
    \]
gives us the first claim, where $ \widetilde{\Omega}_{G,1}$ and $\widetilde{\Omega}_{H,1}$ are the isomorphisms defined in Lemma~\ref{lem:decomposition_of_basic_construction} associated to $L_\omega(G)$ and $L_\omega(H)$, respectively. 

Now, if $v: \H \to L^2(G) \otimes \K$ is a standard intertwiner for a left $(L_\omega(G),\varphi^\omega_G)$-module $(\pi,\H)$, then it follows that $(U_{G,1}^* \otimes 1)v$ is a standard intertwiner for the left  $(L(G(\omega)), \varphi_{G(\omega)})$-module $(\pi \circ \widetilde{\Omega}_{G,1}, \H)$, where $U_{G,1}$ is the unitary defined by equation (\ref{eqn:unitary_of_the_extension}). Additionally, its Murray--von Neumann dimension is then given by 
    \begin{align*}
        \dim_{(L(G(\omega)),\varphi_{G(\omega)})}(\pi \circ \widetilde{\Omega}_{G,1}, \H) &= (\varphi_{G(\omega)} \otimes \Tr_K)\left[ (J_{\varphi_{G(\omega)}}\otimes 1) (U_{G,1}^* \otimes 1)vv^*(U_{G,1} \otimes 1)J_{\varphi_{G(\omega)}}\otimes 1)\right]\\
        &=(\varphi_{G(\omega)} \otimes \Tr_K)\left[ (U_{G,1}^* \otimes 1)(J_{\varphi^\omega_G}\otimes 1) vv^*(J_{\varphi^\omega_G}\otimes 1)(U_{G,1} \otimes 1)\right]\\
        &=((\varphi_{G(\omega)}\circ\Omega_{G,1}^{-1})  \otimes \Tr_K)\left[(J_{\varphi^\omega_G}\otimes 1) vv^*(J_{\varphi^\omega_G}\otimes 1)\right]\\
        &=(\varphi_G^\omega \otimes \Tr_K)\left[ (J_{\varphi^\omega_G}\otimes 1) vv^*(J_{\varphi^\omega_G}\otimes 1)\right],\\
        &=\dim_{(L_\omega(G),\varphi_G^\omega)}(\pi, \H)
    \end{align*}
where $\Omega_{G,1}$ is the isomorphism of Proposition~\ref{prop:decomp_of_central_extension} associated to $L_\omega(G)$. Thus, $(\pi \circ \theta^\omega_\Delta,\H)$ is a left $(L_\omega(H),\varphi^\omega_H)$-module and by the dimension formula of $\theta_\Delta$ in \cite[Theorem 5.8]{GGN25} and the dimension computation above for $L_\omega(G)$ and $L_\omega(H)$, we have
    \begin{align*}
        \dim_{(L_\omega(H),\varphi_H^\omega)}(\pi\circ\theta_\Delta^\omega, \H)&=\dim_{(L(H(\omega)),\varphi_{H(\omega)})}(\pi \circ \theta^\omega_\Delta\circ\widetilde{\Omega}_{H,1}, \H)\\
        &=\dim_{(L(H(\omega)),\varphi_{H(\omega)})}(\pi \circ \widetilde{\Omega}_{G,1}\circ \theta_\Delta, \H)\\
        &=\left(\frac{1}{|\Delta|}\sum_{\delta\in \Delta} \delta\right) [\mu_\mathbb{T} \times \mu_G: \mu_\mathbb{T}\times\mu_H]\dim_{(L(G(\omega)),\varphi_{G(\omega)})}(\pi \circ \widetilde{\Omega}_{G,1}, \H)\\
        &=\left(\frac{1}{|\Delta|}\sum_{\delta\in \Delta} \delta\right) [\mu_G: \mu_H]\dim_{(L_\omega(G),\varphi_{G^\omega})}(\pi, \H),
    \end{align*}
as claimed.
\end{proof}

When $\Delta_H(H) = \Delta_G(G)$, we can choose $\Delta=\{1\}$. Thus we have the following result, since the above scaling factor becomes $[\mu_G:\mu_H]$ and $\theta_\Delta^\omega(e_{\varphi_H^\omega})=e_{\varphi_G^\omega}$.
\begin{cor}
Let $G$ be a second countable almost unimodular group with finite covolume subgroup $H \leq G$, let $\omega: G \times G \to \mathbb{T}$ be a 2-cocycle and let $\varphi^\omega_G$ (resp. $\varphi^\omega_H$) be the twisted Plancherel weight on $L_\omega(G)$ (resp. $L_{\omega}(H)$) associated to a left Haar measure $\mu_G$ on $G$ (resp. $\mu_H$ on $H$). Suppose $\Delta_H(H) =\Delta_G(G)$. Then identifying $\langle L_\omega(H),\varphi_H^\omega\rangle \cong \langle L_\omega(H),\varphi_G^\omega\rangle\leq \langle L_\omega(G),\varphi_G^\omega\rangle$ one has
    \begin{align}\label{eqn:covolume_dimension_formula_equal_modular_groups}
        \dim_{(L_\omega(H),\varphi^\omega_H)}(\pi,\H) = [\mu_G: \mu_H] \dim_{(L_\omega(G),\varphi^\omega_G)}(\pi,\H),
    \end{align}
for any left $(L_\omega(G),\varphi^\omega_G)$-module $(\pi,\H)$.
\end{cor}

The following results shows when a representation of the twisted group von Neumann algebra can be extend to a representation of the basic construction for the inclusion $L_\omega(G)^{\varphi^\omega_G} \leq L_\omega(G)$; that is, the twisted analogue of \cite[Theorem 5.4]{GGN25}. 
\begin{thm}\label{thm:extending_reps_to_twisted_basic_construction}
Let $G$ be a second countable almost unimodular group $G$ with a 2-cocycle $\omega: G\times G \to \mathbb{T}$ and fix a twisted Plancherel weight $\varphi^\omega_G$ on $L_\omega(G)$. For a representation $(\pi, \H)$ of $L_\omega(G)$, the following are equivalent:
    \begin{enumerate}[label=(\roman*)]
        \item There is a representation $\widetilde{\pi}\colon \<L_\omega(G), e_{\varphi^\omega_G}\>\to B(\H)$ extending $\pi$.
        \item There exists a representation $U\colon \Delta_G(G)\hat{\ }\to U(\H)$ satisfying $U_\gamma \pi(\lambda_G^\omega(s)) U_\gamma^* = (\gamma \mid \Delta_G(s)) \pi(\lambda_G^\omega(s))$.
    \end{enumerate}
In particular, given either of the above representations, there is a unique other representation satisfying
    \begin{align}\label{eqn:Jones_projection_image}
        \widetilde{\pi}(e_{\varphi^\omega_G}) = \int_{\Delta_G(G) \hat{\ }} U_\gamma d\mu_{\Delta_G(G) \hat{\ }}(\gamma),
    \end{align}
where $\mu_{\Delta_G(G) \hat{\ }}$ is the unique Haar probability measure on $\Delta_G(G) \hat{\ }$.
\end{thm}
\begin{proof}
Assume that $\omega$ is fully normalized (see discussion before \cite[Lemma 2]{Kle74} and \cite[Proposition 2.4]{Sut80}). Denote $G(\omega) := \mathbb{T} \rtimes_{(1,\omega)}G$ and $K := \Delta_G(G)\hat{\ }$. The representation $(\pi,\H)$ can be lifted to a representation of $L(G(\omega))$ by $(\pi \circ \Omega_{G,1}, \H)$, where $\Omega_{G,1}$ is the normal, unital $*$-homomorphism in Proposition \ref{prop:decomp_of_central_extension}. Thus if (ii) holds, then for each $\gamma \in K$ and $(a,s) \in G(\omega)$, we have 
    \[
        U_\gamma \pi (\Omega_{G,1}(\lambda_{G(\omega)}(a,s)))U_\gamma^* =(\gamma \mid \Delta_{G}(s)) a\pi (\lambda_{G}^\omega(s))= (\gamma \mid \Delta_{G(\omega)}(a,s))\pi (\Omega_{G,1}(\lambda_{G(\omega)}(a,s))).
    \]
By \cite[Theorem 5.4]{GGN25}, we can also extend $(\pi \circ \Omega_{G,1}, \H)$ to a representation $(\pi_\omega,\H)$ of $\langle L(G(\omega)),\varphi_{G(\omega)}\rangle$. Hence $(\pi_\omega \circ \widetilde{\Omega}^{-1}_{G,1}, \H)$ is a representation of $\langle L_\omega(G), \varphi^\omega_G\rangle$ that extends $(\pi,\H)$, where $\widetilde{\Omega}_{G_1}$ is the normal, unital $*$-homomorphism in Lemma~\ref{lem:decomposition_of_basic_construction}. 

Conversely, if $(\widetilde{\pi},\H)$ is an extension to the basic construction $\<L_\omega(G), e_{\varphi^\omega_G}\>$, we can further extend to the basic construction $\<L(G(\omega)), e_{\varphi_{G(\omega)}}\>$ by $(\widetilde{\pi} \circ\widetilde{\Omega}_{G,1},\H)$. Then by \cite[Theorem 5.4]{GGN25}, there exists a representation $U\colon K\to U(\H)$ satisfying $U_\gamma \widetilde{\pi}( \widetilde{\Omega}_{G,1}(\lambda_{G(\omega)}(1,s))) U_\gamma^* = (\gamma \mid \Delta_G(s)) \widetilde{\pi}(\widetilde{\Omega}_{G,1}(\lambda_{G(\omega)}(1,s)))$. Hence (ii) holds.
\end{proof}

In \cite[Theorem 5.12]{GGN25}, we generalize the Atiyah--Schmidt formula for finite covolume subgroups of almost unimodular groups. We provide a slight generalization to the case when the almost unimodular group admits a 2-cocycle $\omega$. By Theorem~\ref{thm:sq_int_irred_proj_reps}, we know that irreducible square integrable $\omega$-projective representation $(\pi,\H)$ of an almost unimodular group $G$ is induced by an irreducible square integrable representation $(\pi_1,\H_1)$ of its unimodular part $\ker\Delta_G$. Furthermore, the square integrability implies that these representations admit extensions to $L_\omega(G)$ and $L_\omega(\ker\Delta_G)$, respectively. By Theorem~\ref{thm:proj_reps_induced_from_kernel} and \ref{thm:extending_reps_to_twisted_basic_construction}, there is an extension $\widetilde{\pi}$ of $\pi$ to the basic construction $\langle L_\omega(G), \varphi^\omega_G\rangle$ satisfying $\widetilde{\pi}(e_{\varphi_G^\omega})\H = \H_1$. We omit the details of the following proof since it is similar to how we computed the dimension in Theorem~\ref{thm:general_covolume_dimension_twisted} and uses the Atiyah--Schmidt formula \cite[Theorem 5.12]{GGN25}.

\begin{thm}\label{thm:Atiya-schmidt_formula_twisted}
Let $G$ be a second countable almost unimodular group with finite covolume subgroup $H\leq G$, let $\omega:G \times G \to \mathbb{T}$ be a 2-cocycle, let $\varphi^\omega_G$ (resp. $\varphi^\omega_H$) be the Plancherel weight on $L_\omega(G)$ (resp. $L_\omega(H)$) associated to a left Haar measure $\mu_G$ on $G$ (resp. $\mu_H$ on $H$), and for each set $\Delta$ of coset representatives of $\Delta_H(H)\leq \Delta_G(G)$ let $\theta^\omega_\Delta\colon \<L_\omega(H),e_{\varphi^\omega_H}\> \to \<L_\omega(G),e_{\varphi^\omega_G}\>$ be as in Theorem~\ref{thm:general_covolume_dimension_twisted}. Let $(\pi, \H)$ be an irreducible square integrable $\omega$-projective representation of $G$, let $(\pi_1,\H_1)$ be the irreducible square integrable $\omega$-projective representation of $\ker{\Delta_G}$ that induces $(\pi,\H)$, and let $(\widetilde{\pi},\H)$ be the representation of $\<L_\omega(G),e_{\varphi^\omega_G}\>$ extending $(\pi,\H)$. Then one has
    \[
        \dim_{(L_\omega(H), \varphi^\omega_H)}(\widetilde{\pi}\circ\theta^\omega_\Delta, \H) = d_{\pi_1} \left(\frac{1}{|\Delta|}\sum_{\delta\in \Delta} \delta\right) [\mu_G : \mu_H],
    \]
where $d_{\pi_1}$ is the formal degree of $(\pi_1, \H_1)$ with respect to $\mu_{\ker{\Delta_G}}:=\mu_G|_{\mathcal{B}(\ker{\Delta_G})}$.
\end{thm}

Finally, we highlight the special case $\Delta_H(H)=\Delta_G(G)$ in the previous theorem as a corollary:

\begin{cor}
Let $G$ be a second countable almost unimodular group with finite covolume subgroup $H\leq G$, let $\omega:G \times G \to \mathbb{T}$ be a 2-cocycle, and let $\varphi^\omega_G$ (resp. $\varphi^\omega_H$) be the Plancherel weight on $L_\omega(G)$ (resp. $L_\omega(H)$) associated to a left Haar measure $\mu_G$ on $G$ (resp. $\mu_H$ on $H$). Suppose $\Delta_H(H)=\Delta_G(G)$ so that we may identify $\<L_\omega(H),e_{\varphi^\omega_H}\>\cong \<L_\omega(H),e_{\varphi^\omega_G}\> \leq \<L_\omega(G),e_{\varphi^\omega_G}\>$. Let $(\pi, \H)$ be an irreducible square integrable $\omega$-projective representation of $G$, let $(\pi_1,\H_1)$ be the irreducible square integrable $\omega$-projective representation of $\ker{\Delta_G}$ that induces $(\pi,\H)$, and let $(\widetilde{\pi},\H)$ be the representation of $\<L_\omega(G),e_{\varphi^\omega_G}\>$ extending $(\pi,\H)$. Then one has
    \[
        \dim_{(L_\omega(H), \varphi^\omega_H)}(\widetilde{\pi}, \H) = d_{\pi_1} [\mu_G : \mu_H],
    \]
where $d_{\pi_1}$ is the formal degree of $(\pi_1, \H_1)$ with respect to $\mu_{\ker{\Delta_G}}:=\mu_G|_{\mathcal{B}(\ker{\Delta_G})}$.
\end{cor}


\bibliographystyle{amsalpha}
\bibliography{references}

\end{document}